\documentclass{amsart}

\usepackage[english]{babel}
\usepackage{amsmath}
\usepackage{graphicx}
\usepackage{amssymb}
\usepackage{amsthm} 	% provides Theorem environment
\usepackage{thmtools} 	% for manual numbering of theorems
\usepackage{amscd} 		% provides CD environment for commutative diagrams (and exact seqs)
\usepackage{dsfont}
\usepackage{fixltx2e}
\usepackage{hyperref} 
\usepackage{tikz-cd}
\usepackage{listings}

\newtheorem*{convention*}{Convention}
\newtheorem{thm}{Theorem}[section]
\newtheorem*{thm*}{Theorem}
\newtheorem{lemma}[thm]{Lemma}
\newtheorem{conj}[thm]{Conjecture}
\newtheorem*{conj*}{Conjecture}
\newtheorem{cor}[thm]{Corollary}
\newtheorem{prop}[thm]{Proposition}
\theoremstyle{remark} \newtheorem{remark}{Remark}
\theoremstyle{definition} \newtheorem{ex}{Example} 
\newtheorem*{defn}{Definition}

\newcommand{\bbC}{\mathbb{C}}

\newcommand{\bbP}{\mathbb{P}}

\newcommand{\bbZ}{\mathbb{Z}}

\newcommand{\calO}{\mathcal{O}}
\newcommand{\calT}{\mathcal{T}}
\newcommand{\calN}{\mathcal{N}}

\newcommand{\calH}{\mathcal{H}}

\newcommand{\calIC}{\mathcal{IC}_X^\bullet}
\newcommand{\calICT}{\mathcal{IC}_{\tau_{m,n,k}}^\bullet}
\newcommand{\calCC}{\mathcal{CC}}
\newcommand{\calF}{\mathcal{F}}

\newcommand{\codim}{\operatorname{codim}}
\newcommand{\rk}{\operatorname{rk}}
\newcommand{\tr}{\operatorname{trace}}

\newcommand{\ind}{\mathds{1}}
\newcommand\numberthis{\addtocounter{equation}{1}\tag{\theequation}}

\begin{document}
\bibliographystyle{plain}

\title{Chern classes and Characteristic Cycles of Determinantal Varieties}
\author[rvt]{Xiping Zhang}

\begin{abstract}
Let $K$ be an algebraically closed field of characteristic $0$.
For $m\geq n$, we define $\tau_{m,n,k}$ to be the set of $m\times n$ matrices 
over $K$ with kernel dimension $\geq k$. This is a projective subvariety 
of $\bbP^{mn-1}$, and is called the (generic) determinantal variety.
In most cases $\tau_{m,n,k}$ is singular with singular locus 
$\tau_{m,n,k+1}$. 
In this paper we give explicit formulas computing
the Chern-Mather class ($c_M$) and the Chern-Schwartz-MacPherson class ($c_{SM}$) of $\tau_{m,n,k}$, 
as  classes in the projective space. 
We also obtain formulas for the conormal cycles and
the characteristic cycles of these varieties, and for their generic Euclidean
Distance degree. Further, when $K=\bbC$, we prove that the 
characteristic cycle of the intersection cohomology sheaf of a determinantal
variety agrees with its conormal cycle (and hence is irreducible).

Our formulas are based on calculations of
degrees of certain Chern classes of the universal bundles over the Grassmannian.
For some small values of $m,n,k$, we use Macaulay2 to exhibit examples of the Chern-Mather classes, the Chern-Schwartz-MacPherson classes and the classes of characteristic cycles 
of $\tau_{m,n,k}$. 

On the basis of explicit computations in low dimensions, 
we formulate conjectures concerning the effectivity of the classes and the 
vanishing of specific terms in the Chern-Schwartz-MacPherson classes of the largest strata
$\tau_{m,n,k}\smallsetminus \tau_{m,n,k+1}$.

The irreducibility of the characteristic cycle of the intersection cohomology sheaf follows 
from the Kashiwara-Dubson's microlocal index theorem, a study of the 
`Tjurina transform' of $\tau_{m,n,k}$,  and the recent computation of the local Euler obstruction of $\tau_{m,n,k}$ .
\end{abstract}
\maketitle

\section{Introduction}
Let $K$ be an algebraically closed field of characteristic $0$.
For $m\geq n$, we define $\tau_{m,n,k}$ to be the set of all $m$ by $n$ matrices over 
$K$ with kernel dimension $\geq k$. This is an irreducible projective subvariety of $\bbP^{mn-1}$, and in most cases $\tau_{m,n,k}$ is singular with singular locus $\tau_{m,n,k+1}$. The varieties $\tau_{m,n,k}$ 
are called (generic) determinantal varieties, and have been the object of intense study. 
(See e.g.,~\cite{PP1},~\cite[\S14.4]{INT},~\cite[Lecture 9]{AG-JH}.)

The computation of invariants of $\tau_{m,n,k}$ is a natural task.
For example, when $K=\bbC$, all varieties $\tau_{n,n,k}$, $k=0,1,\cdots n-1$ have the same 
Euler characteristic. This is because
under the group action $T=(\bbC^*)^{ m+n}$, 
they share the same fixed points set: all matrices with exactly one non-zero entry. This is a discrete 
set of $n^2$ points, and since the Euler characteristic is only contributed by fixed 
points~\cite[Theorem 1.3]{MR3165186}, one has $\chi(\tau_{n,n,k})=n^2$ for all $k$.

In the smooth setting, the Euler characteristic is the degree of the total Chern class. However, in most cases $\tau_{m,n,k}$ is singular, with singular 
locus $\tau_{m,n,k+1}$. There is an extension of the notion of total Chern class
to a singular variety $X$, defined by R.~D.~MacPherson in \cite{MAC} over $\bbC$, and generalized to arbitrary algebraically closed field of characteristic $0$ by G. Kennedy in~\cite{Kennedy}. This is the Chern-Schwartz-MacPherson class, denoted by $c_{SM}(X)$, and also has the property that the degree of $c_{SM}(X)$ equals the Euler characteristic of $X$
without any smoothness assumption on $X$. Two important ingredients MacPherson used to define $c_{SM}$ class are the local Euler obstruction and the Chern-Mather class, denoted by $Eu_X$ and $c_M(X)$ respectively. They were originally defined on $\bbC$, but can be generalized to arbitrary base field. 

In this paper we obtain formulas
(Theorem~\ref{thm; cmalgorithm}, Corollary~\ref{thm; csmalgorithm}) that explicitly 
compute the Chern-Mather class and the Chern-Schwartz-MacPherson class of $\tau_{m,n,k}$ over $K$. The information carried by these classes
is equivalent to the information of the conormal and the characteristic cycles.
We give formulas for these cycles (Prop~\ref{prop; contau}), as an application
of our formulas for the Chern-Mather and Chern-Schwartz-MacPherson classes. In particular, this yields an expression for the polar degrees and the generic Euclidean distance degree (\cite{MR3451425}) of
determinantal varieties (Prop~\ref{prop; ged}). For the determinantal hypersurface, our results indicate that the generic Euclidean Distance degree is related to the number of
normal modes of a Hamiltonian system.
We also prove that the characteristic cycle associated
with the intersection cohomology sheaf of a determinantal variety equals its conormal 
cycle, hence is irreducible (Theorem~\ref{thm; charcycle}). 
In \S\ref{example} we work out several explicit examples in low dimension. These examples
illustrate several patterns in the coefficients of the Chern classes and characteristic
cycles, for which we can give general proofs (\S\ref{S; dual}). The examples also suggest
remarkable effectivity and vanishing properties, which we formulate as precise
conjectures (\S\ref{conjectures}).

In~\cite{PP1} and~\cite{PP2} Adam Parusi\'nski and Piotr Pragacz computed the $c_{SM}$ class
of the rank $k$ degeneracy loci over $\bbC$. In their paper, for a $k$-general morphism of vector bundles
$\varphi\colon F\to E$ over $X$, they define
\[
D_k(\varphi)=\{(x,\varphi)|\dim \ker(\varphi_x\colon F_x\to E_x)\geq k\}\subset Hom(F,E)
\]
and give a formula~\cite[Theorem 2.1]{PP1} for the pushforward of $c_{SM}(D_k(\varphi))$ in $A_*(X)$.
For $X=\bbP^{mn-1}$ the projective space of $m$ by $n$ matrices, $F=\calO^{\oplus n}$
the $\rk n$ trivial bundle, $E=\calO(1)^{\oplus m}$ and $\varphi$ the morphism that acts on fiber over
$x\in X$ as the corresponding matrix, the degeneracy loci $D_k(\varphi)$ equals $\tau_{m,n,k}$.
Therefore the case of determinantal varieties considered in this paper is a particular case of the case
considered in~\cite{PP1}. However, the formulas obtained in this paper are very explicit and can be
implemented directly with a software such as Macaulay2~\cite{M2}, in particularly the package Schubert2.
Further, Parusi\'nski and Pragacz work over the complex numbers; the results in this paper hold over arbitrary algebraically closed fields of characteristic $0$.

We recall the definition of the Chern-Mather class and the Chern-Schwartz-MacPherson class, and 
some basic properties of them in~\S\ref{S; c_*}. In \S\ref{S; determinantalvar} we concentrate on the determinantal varieties, and introduce a resolution $\nu\colon \hat\tau_{m,n,k}\to \tau_{m,n,k}$. This is the resolution used by Tjurina in \cite{Tju}, and we call it the Tjurina transform. This is also the desingularization $Z_r$ used in \cite{PP1} when $K=\bbC$.  Parusi{\'n}ski and Pragacz observe that this is a small desingularization in the proof of \cite[Theorem 2.12]{PP1}, where they compute the Intersection Homology Euler characteristic of degeneracy loci. 

In \S\ref{cmather}, for $k\geq 1$,
we compute the $c_M$ class and the $c_{SM}$ class for $\tau_{m,n,k}$ using the Tjurina transform
$\hat\tau_{m,n,k}$, which may be constructed as a projective bundle over a Grassmannian.
The projective bundle structure gives a way to reduce the computation of $\nu_*c_{SM}(\hat\tau_{m,n,k})$
to a computation in the intersection ring of a Grassmannian. In
Theorem~\ref{thm; cmalgorithm} we give an explicit formula for $\nu_*c_{SM}(\hat\tau_{m,n,k})$,
using the degrees of certain Chern classes over a Grassmannian.
Moreover, in Lemma~\ref{lemma; cmeu} we prove that the pushforward class
$\nu_*(c_{SM}(\hat\tau_{m,n,k}))$ agrees with the Chern-Mather class $c_M(\tau_{m,n,k})$, hence the above formula actually computes $c_M(\tau_{m,n,k})$. The proof of the lemma is based on the functorial property of $c_{SM}$ class and a recent result computing the local Euler obstructions of determinantal varieties (Theorem~\ref{thm; eulerobstruction}). This result is proven for varieties over arbitrary algebraically closed fields by intersection-theoretic methods in~\cite{XP}. For determinantal varieties over $\bbC$, this formula for the local Euler obstruction was obtained earlier by T. Gaffney, N. Grulha and M. Ruas using topological methods~\cite{NG-TG}.

We define $\tau_{m,n,k}^\circ$ to be the open subset $\tau_{m,n,k}\smallsetminus\tau_{m,n,k+1}$ of $\tau_{m,n,k}$. The subsets $\{\tau_{m,n,l}^\circ \colon l\geq k\}$ are disjoint from each other, and they form a stratification of $\tau_{m,n,k}$. The functorial property of $c_{SM}$ classes under the pushforward $\nu_*$ gives a formula relating $\nu_*c_{SM}(\hat\tau_{m,n,k})$ and $c_{SM}(\tau^\circ_{m,n,l})$ for $l=k,k+1,\cdots,n-1$. This relation hence provides a way to obtain $c_{SM}(\tau_{m,n,k})$ in terms of $c_M(\tau_{m,n,l})$, $l=k,k+1,\cdots,n-1$, as shown in Corollary~\ref{thm; csmalgorithm}.

For a smooth ambient space $M$, the information of the $c_M$ and $c_{SM}$ class of a subvariety $X$ is essentially 
equivalent to the information of certain Lagrangian cycles in the total space $T^*M$, namely the  conormal cycle $T_X^*M$ of $X$ and the 
characteristic cycle $Ch(X)$ of $X$ (associated with the constant function $1$
on $X$).
In \S\ref{S; projconormal} we review the basic definition of Lagrangian cycles, and apply our results to obtain formulas for the (projectivized) conormal and characteristic cycles of $\tau_{m,n,k}$, as classes in $\bbP^{mn-1} \times \bbP^{mn-1}$ (Proposition~\ref{prop; contau}). The coefficients of the conormal cycle class are also the degrees of the polar classes. Hence our formula also computes the polar degrees of determinantal varieties $\tau_{m,n,k}$. In particular, we get an expression for the generic Euclidean degree of $\tau_{m,n,k}$.
We also include formulas for the characteristic cycles of the strata
$\tau^\circ_{m,n,k}$. As observed in~\S\ref{example}, the classes for these strata appear to satisfy interesting effective and vanishing properties.

When the base field is $\bbC$, Lagrangian cycles may also be associated to complexes of sheaves on a 
nonsingular variety $M$ which are constructible with respect to a Whitney stratification
$\{S_i\}$. In particular, when $X=\overline{S_i}$ is the closure of a stratum, we can consider the intersection cohomology sheaf $\calIC$.
This is a constructible sheaf on $M$, and we call the corresponding
characteristic cycle the `IC characteristic cycle' of $X$, denoted by $\calCC(\calIC)$. 
(For more details about the intersection cohomology sheaf and intersection homology, 
see~\cite{G-M1} and~\cite{G-M2}.) The $IC$ characteristic cycle of $X$ can be 
expressed as a linear combination of the conormal cycles of the strata:
\[
\calCC(\calIC)=\sum_{i\in I} r_i(\calIC)[T^*_{\overline{S_i}}M]\/.
\]
Here the integer coefficients  $r_i(\calIC)$ are called the \textit{Microlocal Multiplicities}. 
(See~\cite[Section 4.1]{Dimca} for an explicit construction of the
$IC$ characteristic cycle and the microlocal multiplicities.)  
Using our study of the Tjurina transform and knowledge of the local Euler
obstruction, the Dubson-Kashiwara index formula allows us to show that all but
one of the microlocal multiplicities vanish. Therefore in Theorem~\ref{thm; charcycle}
we establish the following result.
\begin{thm*}
The characteristic cycle associated with the
intersection cohomology sheaf of $\tau_{m,n,k}$ equals the conormal cycle
of $\tau_{m,n,k}$; hence, it is irreducible.
\end{thm*}

As pointed out in \cite[Rmk 3.2.2]{Jones}, the irreducibility of the IC characteristic cycle 
is a rather unusual phenomenon. It is known to be true for Schubert varieties in a Grassmannian, for certain Schubert varieties in flag manifolds of types B, C, and D, and for theta divisors of Jacobians (cf. \cite{MR1642745}). 
By the above result, determinantal varieties also share this 
property. 

\S\ref{example} is devoted to explicit examples, leading to the formulation
of conjectures on the objects studied in this paper. Several of the examples have
been worked out by using Macaulay2. 
In \S\ref{S; dual} we highlight some patterns suggested by the examples, and show that
they follow in general from duality considerations.
The examples also suggest interesting vanishing properties on 
the low dimensional terms of the Chern-Schwartz-MacPherson classes: 
we observe that the coefficients of the $l$-dimensional pieces of 
$c_{SM}(\tau_{m,n,k}^\circ)$ vanish when $l\leq(n-k-2)$. 
Another attractive observation we make here is that all the nonzero coefficients 
appearing in the $c_M$ classes and the $c_{SM}$ classes are positive. 
In~\S\ref{conjectures} we propose precise conjectures about the effectivity property (Conjecture~\ref{conjecture; effective}) and the vanishing property 
(Conjecture~\ref{conjecture; vanishing}) for determinantal varieties in projective space. 
These facts call for a conceptual, geometric explanation. 

The situation appears to have similarities with the case of Schubert varieties.
In~\cite{MR3419959} June Huh proved that the $c_{SM}$ class of Schubert varieties 
in Grassmanians are effective, which was conjectured earlier by Paolo Aluffi and 
Leonardo Mihalcea (\cite{Aluffi4}). The $c_{SM}$ class of Schubert varieties 
in any flag manifold are also conjectured to be effective (\cite{Aluffi3}), and this
conjecture is still open.

When $K=\bbC$, the torus action by $(\bbC^*)^{m+n}$ on $\tau_{m,n,k}\subset \bbP^{mn-1}$ leads to the equivariant version of the story. In fact, some examples also indicate similar vanishing property for the low dimensional components of the equivariant Chern-Schwartz-MacPherson class. These results will be presented in a following paper.

I would like to thank Paolo Aluffi for all the help and support, and Terence Gaffney and Nivaldo G. Grulha Jr for the helpful discussion during my visit to Northeastern University. I would like to thank Corey Harris for teaching me how to use Macaulay2.

\section{Chern-Mather and Chern-Schwartz-MacPherson Class}
\label{S; c_*}
\begin{convention*}
In this paper unless specific described, all the varieties are assumed to be irreducible, and all subvarieties are closed irreducible subvarieties. 
\end{convention*}
Let $K$ be an algebraically closed field of characteristic $0$. 
Let $X$ be an algebraic variety over $K$.
A subset of $X$ is {\em constructible\/} if it can be
obtained from subvarieties of $X$ by finitely many ordinary set-theoretical operations.
A {\em constructible function} on $X$ is an integer-valued function $f$ such that there is a decomposition of $X$ as a finite union of constructible
sets, for which the restriction of $f$ to each subset in the decomposition is constant.
Equivalently, a constructible function on $X$ is a finite sum $\sum_W m_W \ind_W$ over all closed subvarieties $W$ of $X$, and the coefficients $m_W$ are integers.
Here $\ind_W$ is the indicator function that evaluates to $1$ on $W$ and to $0$ 
in the complement of $W$.

We will denote by $F(X)$ the set of constructible functions on $X$. This is an abelian group under addition of functions, and is freely generated by the indicator functions $\{\ind_W\colon W\subset X \textrm{ is a closed subvariety}\}$.

Let $f\colon X\to Y$ be a proper morphism. One defines a homomorphism $Ff\colon F(X)\to F(Y)$ by setting, for all $ p\in Y$,
\begin{equation*}
Ff(\ind_W)(p) = \chi(f^{-1}(p)\cap W)
\end{equation*} 
and extending this definition by linearity. One can verify that this makes $F$ into 
a functor from the category $\mathcal{VAR}=$
\{algebraic varieties over $K$, proper morphisms\} to the category $\mathcal{AB}$ of abelian groups.

There is another important functor from $\mathcal{VAR}$ to $\mathcal{AB}$, which is the Chow group functor $A$ with $Af\colon A_*(X)\to A_*(Y)$ being the pushforward of 
rational equivalent classes of cycles. It is natural to compare the two functors, 
and study the maps between them.  For compact complex varieties, the existence and uniqueness of the natural transformation from $F$ to $A$ that normalize to the total Chern class on smooth varieties was conjectured
by Deligne and Grothendieck, and was proved by R. D. MacPherson in 1973~\cite{MAC}. Then in 1990, G. Kennedy generalized the result to arbitrary algebraically closed field of characteristic $0$ in \cite{Kennedy}.
\begin{thm}[R. D. MacPherson, 1973]
\label{thm; macc_*}
Let $X$ be a compact complex variety
There is a unique natural transformation $c_*$ from the functor
$F$ to the functor $A$ such that if $X$ is smooth, then 
$c_*(\ind_X)=c(\mathcal{T}_X)\cap [X]$, where $\mathcal{T}_X$ 
is the tangent bundle of $X$.
\end{thm}
\begin{thm}[G. Kennedy, 1990]
The above theorem can be generalized to arbitrary algebraically closed field of characteristic $0$.
\end{thm}

For complex varieties, in \cite{MAC} MacPherson defined two important concepts as the main ingredients in his definition of $c_*$: the local Euler obstruction function and the Chern-Mather class. Let $XM$ be a compact complex variety. For any subvariety $V\subset X$, the local Euler obstruction function $Eu_V$ is defined as the local obstruction to extend certain $1$-form around each point. This is an integer-valued function on $X$, and measures $0$ outside $V$. 
He also defined the Chern-Mather class of $V$ as a cycle in $A_*(V)$ denoted by $c_M(V)$. It is  the pushforward of the Chern class of the Nash tangent bundle from the Nash Blowup of $V$ (See \cite{MAC} for precise definitions). 

The proof of the theorem can be summarized as the following steps. 
\begin{enumerate}
\item For any subvariety $W\subset X$, $Eu_W$ is a constructible function, i.e., $Eu_W=\sum_{V} e_Z \ind_Z$ for some sub-varieties $Z$ of $X$.
\item $\{Eu_W| W \text{ is a subvariety of } X\}$ form a basis for $F(X)$.
\item Let $i\colon W\to X$ be the closed embedding. Define $c_*(Eu_W)=i_*(c_M(W))$ to be the pushforward of the Chern-Mather class of $W$ in $A_*(X)$. This is the unique natural transformation that matches the desired normalization property. 
\end{enumerate}

\begin{remark}
\begin{enumerate}
\item
Another definition of Chern classes on singular varieties is due to M.-H. Schwartz, who uses obstruction theory and radial frames to construct such classes. Details of the construction can be found in~\cite{MR32:1727}~\cite{MR35:3707}~\cite{MR629125}. Also in~\cite{MR629125} it is shown that these classes correspond, by Alexander isomorphism, to the classes defined by MacPherson in the above theorem. 
\item  MacPherson's original work was on homology groups, but one can change settings and get a Chow group version of the theorem. Cf.~\cite[19.1.7]{INT}. 
\item The original definitions of the Local Euler obstruction and Chern-Mather class made by MacPherson were for complex varieties, but one can extend them to arbitrary algebraically closed base field $K$ (Cf. \cite{MR629121}). 
\item
Assuming that $X$ is compact, if we consider the constant map $k\colon X\to \{p\}$, then the covariance property of $c_*$
shows that
\begin{align*}
\int_X c_{SM}(Y)
=&~ \int_{\{p\}} Afc_*(\ind_Y)=\int_{\{p\}} c_*Ff(\ind_Y)\\
=&~ \int_{\{p\}} \chi(Y)c_*(\ind_{\{p\}})=\chi(Y). \qedhere
\end{align*}
This observation gives a generalization of the classical
Poincar\'e-Hopf Theorem to possibly singular varieties.
\end{enumerate}
\end{remark} 

For algebraic varieties over algebraically closed field $K$, the important ingredients used in Kennedy's proof are the conormal and Lagrangian cycles. We will give more details in \S\ref{S; projconormal}. The generalization in \cite{Kennedy} 
can be summarized as the following steps.  Let $X$ be an algebraic variety over $K$, and assume that $X\subset M$ is a closed embedding into some compact smooth ambient space.
\begin{enumerate}
\item Let $L(X)$ be the group of Lagrangian cycles of $X$, i.e., the free abelian group on the set of conical Lagrangian subvarieties in $T^*M|_{X}$.
 This is a functor from the category of algebraic varieties with proper morphisms to abelian groups.
\item The functor $F$ of constructible functions is isomorphic to the functor $L$ by taking $Eu_V$ to the conormal cycle of $V$.
\item There is a unique natural transform $c_*$ from $L$ to the Chow group functor $A$ that is compatible with the natural transform defined by MacPherson's over $\bbC$, under the identification of $L$ and $F$. 
\end{enumerate}

\begin{defn}
Let $X$ be an ambient space over $K$. Let $Y\subset X$ be a locally closed subset of $X$. 
Hence $\ind_Y$ is a constructible function in $F(X)$, and the class $c_*(\ind_Y)$ in $A_*(X)$ is called the Chern-Schwartz-MacPherson class of~$Y$, denoted by $c_{SM}(Y)$. We will let $c_M(Y)$ be the Chern-Mather class of $Y$, and we will usually implicitly view this class as an element
of $A_*(X)$.
\end{defn}

In this paper, we consider $Y=\tau_{m,n,k}$ to be our studying object, and pick $X=\bbP^{mn-1}$ to be the ambient space. Note that the Chow group of $\bbP^{mn-1}$ is the free $\bbZ$-module generated by $\{1,H,H^2,\cdots ,H^{mn-1}\}$, where $H=c_1(\calO(1))$ is the hyperplane class. It also forms an intersection ring under intersection product, which is $\bbZ[H]/(H^{mn})$. Hence by definition $c_M(\tau_{m,n,k})$ and $c_{SM}(\tau_{m,n,k})$ can be expressed as polynomials of degree $\le mn$ in $H$.

\section{Determinantal Varieties and the Tjurina Transform}
\label{S; determinantalvar}
\subsection{Determinantal Variety}
Let $K$ be an algebraically closed field of characteristic $0$. For $m\geq n$, let $M_{m,n}=M_{m,n}$ be the set of $m \times n$ nonzero
matrices over $K$ up to scalar. We view this set as a projective space
 $\bbP^{mn-1}=\bbP(Hom(V_n,V_m))$ for some $n$-dimensional vector space $V_n$ and $m$-dimensional vector space $V_m$ over $K$. 
For $0\leq k\leq n-1$, we consider the subset $\tau_{m,n,k} \subset M_{m,n}$ consisting of all the matrices whose kernel has dimension no less than $k$, or equivalently with rank no bigger than $n-k$. Since the rank condition is equivalent to the vanishing of all $(n-k+1)\times(n-k+1)$ minors, $\tau_{m,n,k}$ is a subvariety of $\bbP^{mn-1}$. The varieties $\tau_{m,n,k}$ are called \textit{(generic) Determinantal Varieties}.

The determinantal varieties have the following basic properties:
\begin{enumerate}
\item When $k=0$, $\tau_{m,n,0}=\bbP^{mn-1}$ is the whole porjective space.
\item When $k=n-1$, $\tau_{m,n,n-1}\cong \bbP^{m-1}\times \bbP^{n-1}$ is isomorphic to the Segre embedding.
\item $\tau_{m,n,k}$ is irreducible, and $\dim\tau_{m,n,k}=(m+k)(n-k)-1$.
\item For $i\leq j$, we have the natural closed embedding $\tau_{m,n,j}\hookrightarrow \tau_{m,n,i}$. In particular, for $j=i+1$, we denote the open subset $\tau_{m,n,i}\smallsetminus \tau_{m,n,i+1}$ by $\tau_{m,n,i}^\circ$.
\item For $n-1>k\geq 1$, and $n\geq 3$, the varieties $\tau_{m,n,k}$ are singular with singular locus $\tau_{m,n,k+1}$. Hence $\tau_{m,n,k}^\circ$ is the smooth part of $\tau_{m,n,k}$. We make the convention here that $\tau_{m,n,n-1}^\circ=\tau_{m,n,n-1}$.
\item For $i=0,1,\cdots ,n-1-k$, the subsets $\tau_{m,n,k+i}^\circ$ form a disjoint decomposition of $\tau_{m,n,k}$. When $k=\bbC$, this is a Whitney stratification.
\end{enumerate}

\subsection{The Tjurina Transform}
\label{Tjurina}
In this section we introduce a resolution $\hat \tau_{m,n,k}$ of $\tau_{m,n,k}$, which will be used later in the computation of the Chern-Mather class. The resolution was used by Tjurina in \cite{Tju}, and we call it the Tjurina transform. It is defined as the incidence correspondence in $G(k,n)\times\bbP^{mn-1}$:
\begin{align*}
\hat \tau_{m,n,k}:=\{(\Lambda,\varphi)|\varphi\in \tau_{m,n,k};  \Lambda\subset\ker\varphi\}.
\end{align*}
And one has the following diagram:
\begin{small}
\[
\begin{tikzcd}
 & \hat\tau_{m,n,k} \arrow{r}{} \arrow{d}{\nu} \arrow[swap]{dl}{\rho} & G(k,n)\times \bbP^{mn-1} \arrow{d} \\
G(k,n)  & \tau_{m,n,k} \arrow{r}{i} & \bbP^{mn-1}.
\end{tikzcd}
\]
\end{small}

\begin{prop} 
The map $\nu\colon \hat\tau_{m,n,k}\to \tau_{m,n,k}$ is birational. Moreover, $\hat\tau_{m,n,k}$ is isomorphic to the projective bundle $\bbP(Q^{\vee m})$ over $G(k,n)$, where $Q$ denotes the 
universal quotient bundle of the
Grassmanian. In particular $\hat\tau_{m,n,k}$ is smooth, therefore it is a resolution of singularity of $\tau_{m,n,k}$.
\end{prop}
\begin{proof}
Let $Q$ be the universal quotient bundle on $G(k,n)$. For any $k$-plane $\Lambda\in G(k,n)$, the fiber of $\rho$ is $\rho^{-1}(\Lambda)=
\{\varphi\in \bbP^{mn-1} |\Lambda\subset ker\varphi\}$. Consider the space of linear morphisms $Hom(V_n/\Lambda,V_m)=(Q^{\vee m})|_{\Lambda}$. The fiber over $\Lambda$ is isomorphic to the projectivization of $(Q^{\vee m}))|_{\Lambda}$ by factoring through a quotient. 
\begin{align*}
\varphi\colon V_n\to V_m \mapsto \bar{\varphi}\colon V_n/\Lambda\to V_m
\end{align*}
This identifies $\hat\tau_{m,n,k}$ to the projective bundle $\bbP(Q^{\vee m})$, whose rank is $m(n-k)-1$. Hence $\dim \hat\tau_{m,n,k}=m(n-k)-1+k(n-k)=(m+k)(n-k)-1$.

To show that $\nu$ is birational, we consider the open subset $\tau_{m,n,k}^{\circ}$. 
Note that $\nu^{-1}(\tau_{m,n,k+1})$ is cut out from $\hat\tau_{m,n,k}$ by the $(k+1)\times(k+1)$ 
minors, so its complement $\nu^{-1}(\tau_{m,n,k}^{\circ})$ is indeed open. 
For every $\varphi\in\tau_{m,n,k}^{\circ}$, one has $\dim\ker\varphi=k$, hence the fiber $\nu^{-1}(\varphi)=\ker\varphi$ contains exactly one point. 
This shows that $\dim\hat\tau_{m,n,k} = 
\dim\tau_{m,n,k}=(m+k)(n-k)-1$, which implies that $\nu$ is birational.
\end{proof}
 
Since $\hat\tau_{m,n,k}$ is the projective bundle $\bbP(Q^{\vee m})$, we have the 
following Euler sequence:
\[
\begin{tikzcd}
0 \arrow{r} & \calO_{\hat\tau_{m,n,k}} \arrow{r} & \rho^*(Q^{\vee m})\otimes \calO_{\hat\tau_{m,n,k}}(1)\arrow{r} & \mathcal{T}_{\hat\tau_{m,n,k}}\arrow{r} & \rho^*\mathcal{T}_{G(k,n)} \arrow{r} & 0
\end{tikzcd}
\numberthis \label{eulersequence}
\]

\begin{remark}
Since $\hat\tau_{m,n,k}$ is smooth, then its $c_{SM}$ class is just the ordinary total chern class $c(\calT_{\hat \tau_{m,n,k}})\cap [\hat \tau_{m,n,k}]$. Moreover, it is the projective bundle $\bbP(Q^{\vee m})$ over the Grassmanian. As we will see in the following sections,
this reduces the computation of $c_{SM}(\tau_{m,n,k})$ to a computation in the Chow ring of a Grassmannian. 
\end{remark} 

Moreover, the Tjurina transform $\hat\tau_{m,n,k}$ is a small resolution of $\tau_{m,n,k}$. First let's recall the definition of a small resolution:
\begin{defn}
Let $X$ be a irreducible algebraic variety. Let $p\colon Y\to X$ be a resolution of singularities. $Y$ is called a \textit{Small Resolution} of $X$ if for all $i>0$, 
one has $\codim_X \{x\in X| \dim p^{-1}(x)\geq i\} > 2i$.
\end{defn}
 
\begin{prop}
\label{prop; small}
The Tjurina transform $\hat \tau_{m,n,k}$ is a small resolution of $\tau_{m,n,k}$.
\end{prop}
\begin{proof}
Define $L_i=\{x\in \tau_{m,n,k}| \dim \nu^{-1}(x)\geq i \}$. We just need to show that $\codim_{\tau_{m,n,k}} L_i >2i$. Notice that for any $p\in \tau_{m,n,j}^\circ\subset \tau_{m,n,k}$, we have
$\nu^{-1}(p)=\{(p,\Lambda)|\Lambda\subset \ker p\}\cong G(k,j)$.
Hence $\dim \nu^{-1}(p)=k(k-j)$ for any $p\in \tau_{m,n,j}^\circ$, and $L_i=\tau_{m,n,s}$, where $s=k+\lfloor\frac{i}{k}\rfloor$.
So we have
\begin{align*}
\codim_{\tau_{m,n,k}} L_i
=& \codim_{\tau_{m,n,k}} \tau_{m,n,s} \\
=& (m+k)(n-k)-(m+s)(n-s) \\
=& (m+k)(n-k)-(m+k+\lfloor\frac{i}{k}\rfloor)(n-k+\lfloor\frac{i}{k}\rfloor) \\
=&  \lfloor\frac{i}{k}\rfloor((m+k)-(n-k)+\lfloor\frac{i}{k}\rfloor) \\
=&  \lfloor\frac{i}{k}\rfloor(m-n+2k+\lfloor\frac{i}{k}\rfloor) 
> 2i \/. \qedhere
\end{align*}
\end{proof}
\begin{remark}
The resolution $\hat\tau_{m,n,k}$ agrees with the desingularization $Z_r$ defined in \cite[p.804, Diagram 2.1]{PP1} when $K=\bbC$.  Parusi{\'n}ski and Pragacz observed that this is a small desingularization in the proof of \cite[Theorem 2.12]{PP1}, where they compute the Intersection Homology Euler characteristic of degeneracy loci.
\end{remark}

\section{The Main Algorithm}
\label{cmather}
\subsection{Chern-Mather Class via Tjurina Transform}
\label{S; CMSR}
The Chern-Mather class of a variety is defined in terms of its Nash blow-up
and Nash tangent bundle. However, the recent result on the Local Euler obstruction of determinantal varieties suggests that we can directly use the Tjurina transform $\hat\tau_{m,n,k}$ to compute $c_M(\tau_{m,n,k})$.

\begin{thm}[Theorem 5 \cite{XP}]
\label{thm; eulerobstruction}
Let $K$ be an algebraically closed field.
Let $\tau_{m,n,k}$ be the determinantal variety over $K$. For any $\varphi\in\tau_{m,n,k+i}\smallsetminus\tau_{m,n,k+i+1}$, the local Euler obstruction of $\tau_{m,n,k}$ at $\varphi$ equals
\[
Eu_{\tau_{m,n,k}}(\varphi)=\binom{k+i}{i}.
\]
\end{thm}
\begin{remark}
Over $\bbC$, this formula for the local Euler obstruction was obtained
earlier by Gaffney, Grulha, and Ruas in \cite{NG-TG}.
They worked with the affine determinantal variety $\Sigma_{m,n,k}$ , that is, the affine cone over $\tau_{m,n,k}$. It is easy to see that this does not affect the local Euler obstruction, since $\Sigma_{m,n,k}$ locally is the product of $\tau_{m,n,k}$ with $\bbC^*$.
\end{remark}

Use the above result, one can prove the following Lemma.
\begin{lemma}
\label{lemma; cmeu}
Let $\nu \colon \hat\tau_{m,n,k}\to \tau_{m,n,k}$ be the Tjurina transform of $\tau_{m,n,k}$. Then the Chern-Mather class of $\tau_{m,n,k}$ equals
\begin{align*}
c_M(\tau_{m,n,k})
&=\nu_*(c_{SM}(\hat\tau_{m,n,k})) \\
&=\nu_*(c(T_{\hat\tau_{m,n,k}})\cap [\hat\tau_{m,n,k}]) \/.
\end{align*}
\end{lemma}
\begin{proof}
Let $c_*$ be the natural transformation defined in \S\ref{S; c_*}.
For any $\varphi\in\tau_{m,n,k+i}^\circ$, the fiber of $\hat\tau_{m,n,k}$ at $\varphi$ is $\nu^{-1}(\varphi)\cong G(k,k+i)$. By the above theorem, the local Euler obstruction of $\tau_{m,n,k}$ at $\varphi$ equals
\[
Eu_{\tau_{m,n,k}}(\varphi)=\binom{k+i}{k}=\chi(G(k,k+i)) \/.
\]
Hence we have 
\begin{align*}
\nu_*(\ind_{\hat\tau_{m,n,k}})
&=\sum_{i=0}^{n-1-k} \chi(G(k,k+i)) \ind_{\tau_{m,n,k+i}^\circ}\\
&= \sum_{i=0}^{n-1-k} \binom{k+i}{k} \ind_{\tau_{m,n,k+i}^\circ} \\
&= Eu_{\tau_{m,n,k}} \/.
\end{align*}
Recall that for any variety $X$, $c_M(X)=c_*(Eu_X)$. By the functorial property of $c_*$ one gets
\begin{align*}
c_M(\tau_{m,n,k})
&=\nu_* c_*(\ind_{\hat\tau_{m,n,k}}) \\
&=\nu_*(c_{SM}(\hat\tau_{m,n,k}))  \\
&=\nu_*(c(T_{\hat\tau_{m,n,k}})\cap [\hat\tau_{m,n,k}]) \/. \qedhere
\end{align*}
\end{proof}

\subsection{Main Formula}
Let $N=mn-1$. We recall that the Chow ring of $\bbP^N$ may be realized as
$\bbZ[H]/(H^{N+1})$, where $H$ is the hyperplane class 
$c_1(\calO(1))\cap [\bbP^N]$. The Chern-Mather class $c_M(\tau_{m,n,k})=\nu_*c_{SM}(\hat\tau_{m.n.k})$ then admits the form of a polynomial in $H$. We denote this polynomial by 
\begin{align*}
\Gamma_{m,n,k}=\Gamma_{m,n,k}(H)=\sum_{l=0}^N \gamma_l H^l.
\numberthis \label{eq: gamma}
\end{align*}
Here $\gamma_l=\gamma_l(m,n,k)$ and $\Gamma_{m,n,k}(H)$ are also functions of $m,n,k$. 

We denote $S$ and $Q$ to be the universal sub and quotient bundle over the Grassmanian $G(k,n)$. As shown in~\cite[Appendix B.5.8]{INT} , the tangent bundle of $G(k,n)$ can be identified as 
\[
\calT_{G(k,n)}=Hom(S,Q)=S^\vee\otimes Q .
\]
For $k\geq 1$, $i,p=0,1\cdots m(n-k)$, we define the following integers
\begin{align*}
A_{i,p}(k)= A_{i,p}(m,n,k) &:=\int_{G(k,n)} c(S^\vee\otimes Q)c_i(Q^{\vee m})c_{p-i}(S^{\vee m}) \cap [G(k,n)] \\
B_{i,p}(k)=B_{i,p}(m,n,k) &:=\binom{m(n-k)-p}{i-p}
\end{align*}
and let
\[
A(k)=A(m,n,k)=\left[ A_{i,p}(k) \right]_{i,p}\quad, \quad B(k)=B(m,n,k)=\left[ B_{i,p}(k) \right]_{i,p}
\]
be $m(n-k)+1 \times m(n-k)+1$ matrices. Here we assume $\binom{a}{b}=0$ for $a<b$ or $a<0$ or $b<0$. 
\begin{remark}
Since the Chern classes of the universal bundles on a Grassmannian are already known, $A_{i,p}(k)$ can be easily computed by Schubert calculus. 
One can use the Schubert2 package 
in Macaulay2~\cite{M2} to compute the numbers. 
\end{remark}
\begin{ex}
\label{ex; 331}
Let $m=n=3$, and $k=1$. One has:
\[
A(3,3,1)=
\begin{bmatrix}
3 & 9 & 3 & 0 & 0 & 0 & 0 \\
0 &-9 &-9 & 0 & 0 & 0 & 0 \\
0 & 0 & 6 & 0 & 0 & 0 & 0 \\
0 & 0 & 0 & 0 & 0 & 0 & 0 \\
0 & 0 & 0 & 0 & 0 & 0 & 0  \\
0 & 0 & 0 & 0 & 0 & 0 & 0  \\
0 & 0 & 0 & 0 & 0 & 0 & 0  \\
\end{bmatrix} . 
\]  
Let $m=4$, $n=3$ and $k=2$. One has:
\[
A(4,3,2)=
\begin{bmatrix}
3 &  12 &  10 & 0 & 0 \\
0 & -12 & -16 & 0 & 0 \\ 
0 & 0 & 6 & 0 & 0 \\
0 & 0 & 0 & 0 & 0 \\
0 & 0 & 0 & 0 & 0 
\end{bmatrix} .
\]  

\end{ex}

The following theorem gives a formula to explicitly compute the Chern-Mather class of $\tau_{m,n,k}$:
\begin{thm}[Main Formula] 
\label{thm; cmalgorithm}
Assume that $n\leq m$, and consider the $m(n-k)+1 \times m(n-k)+1$ matrix 
\[
\calH(k)=\calH(m,n,k)=[H^{mk+j-i}]_{ij},
\]
 where $H$ is the hyperplane class in $\bbP^{mn-1}$.
For $k\geq 1$, the pushforward of $c_M(\tau_{m,n,k})$ to $\bbP^{mn-1}$ equals
\[
c_M(\tau_{m,n,k})=\tr(A(k)\cdot \calH(k)\cdot B(k)).
\]
\end{thm}

Applying the above theorem, one obtains an expression
\begin{equation}\label{eq; betacm}
c_M(\tau_{m,n,k}) = \sum_{l=0}^{mn-1} \beta_l H^{mn-1-l}
\end{equation}
where the coefficients $\beta_l$ are explicit (albeit complicated) expressions
involving degrees of classes in the Grassmannian $G(k,n)$.
\begin{ex}
When $m=n$, $k=1$, based on the explicit computation of $A_{i,p}$, the Chern-Mather class of $\tau_{n,n,1}$ equals
\begin{align*}
c_M(\tau_{n,n,1})
=\sum_{l=0}^{n^2-1} \sum_{i,p=0}^{n-1} (-1)^i \binom{n}{p+1}\binom{n+i-1}{i}\binom{n}{p-i}\binom{n(n-1)-i}{l-n+p-i} H^l .
\end{align*}
One will observe that $\beta_0=0$ and $ \beta_1=n$.
\end{ex}

The following corollary gives a formula for the Chern-Schwartz-MacPherson class of $\tau_{m,n,k}$.
\label{section; csm}
\begin{cor}[Main Formula II]
\label{thm; csmalgorithm}
\begin{enumerate}
\item
The Chern-Schwartz-MacPherson class $c_{SM}(\tau_{m,n,k})$ is given by
\[
c_{SM}(\tau_{m,n,k})=\sum_{i=0}^{n-1-k}(-1)^{i}\binom{k+i-1}{k-1} c_M(\tau_{m,n,k+i}).
\]
When $k=0$, $\tau_{m,n,0}=\bbP^{mn-1}$.
\item
Recall that $\tau_{m,n,k}^{\circ}:=\tau_{m,n,k}\smallsetminus\tau_{m,n,k+1}$. Its $c_{SM}$ class is given by
\[
c_{SM}(\tau_{m,n,k}^{\circ})
= \sum_{i=0}^{n-1-k} (-1)^i\binom{k+i}{k} c_M(\tau_{m,n,k+i}).
\]
\end{enumerate}
\end{cor}
\begin{proof}
Recall that the fiber of $\hat\tau_{m,n,k}$ at $\varphi\in\tau_{m,n,k+i}^\circ$ is $\nu^{-1}(\varphi)\cong G(k,k+i)$, and $\chi(G(k,n))=\binom{n}{k}$. Hence the functorial property of $c_{SM}$ class shows that:
\[
\left( \begin{array}{c}
\nu_*(c_{SM}(\hat\tau_{m,n,k}))\\
\nu_*(c_{SM}(\hat\tau_{m,n,k+1}))\\
\cdots \\
\nu_*(c_{SM}(\hat\tau_{m,n,n-1}))
  \end{array} \right) =
 \left( \begin{array}{cccc}
\binom{k}{k} & \binom{k+1}{k} & \cdots & \binom{n-1}{k} \\
0 & \binom{k+1}{k+1} & \cdots & \binom{n-1}{k+1} \\
\cdots & \cdots & \cdots & \cdots \\
0 & 0 & \cdots & \binom{n-1}{n-1} 
 \end{array} \right) 
\left( \begin{array}{c}
c_{SM}(\tau^{\circ}_{m,n,k})\\
c_{SM}(\tau^{\circ}_{m,n,k+1})\\
\cdots \\
c_{SM}(\tau^{\circ}_{m,n,n-1})
\end{array} \right) .
\] 	
Also, the lemma~\ref{lemma; cmeu} shows that $\nu_*(c_{SM}(\hat\tau_{m,n,i}))=c_M(\tau_{m,n,i})$. Hence one will get the desired formula after inverting the binomial matrix. The inverse matrix is given by the following lemma.
\begin{lemma}
\[
 \left( \begin{array}{cccc}
\binom{k}{k} & \binom{k+1}{k} & \cdots & \binom{n-1}{k} \\
0 & \binom{k+1}{k+1} & \cdots & \binom{n-1}{k+1} \\
\cdots & \cdots & \cdots & \cdots \\
0 & 0 & \cdots & \binom{n-1}{n-1} 
 \end{array} \right) 
\times 
\left( \begin{array}{cccc}
(-1)^0\binom{k}{k} & (-1)^1\binom{k+1}{k} & \cdots & (-1)^{n-1-k}\binom{n-1}{k} \\
0 & (-1)^0\binom{k+1}{k+1} & \cdots & (-1)^{n-2-k}\binom{n-1}{k+1} \\
\cdots & \cdots & \cdots & \cdots \\
0 & 0 & \cdots & (-1)^0\binom{n-1}{n-1} 
 \end{array} \right) 
=I
\] 	
\end{lemma}
To prove the lemma, one takes the $i$th row of the left:
\[
r_i=[\overbrace{0,\cdots,0}^{i-1}, \binom{k+i-1}{k+i-1}, \binom{k+i}{k+i-1}, \cdots, \binom{n-1}{k+i-1}]; 
\]
and the $j$th column of the right
\[
c_j=[(-1)^{j-1} \binom{k+j-1}{k},  (-1)^{j-2}\binom{k+j-1}{k+1},  \cdots,    (-1)^{j-j}\binom{k+j-1}{k+j-1}, 0,\cdots, 0 ]^t.                  
\]
When $i=j$, then $r_i\cdot c_j=(-1)^{j-j}\binom{k+j-1}{k+j-1} \binom{k+i-1}{k+i-1}=1$.
When $i> j$, one can easily observe that $r_i\cdot c_j=0$. 
When $i<j$, the dot product $r_i\cdot c_j$ is given by
\begin{align*}
r_i\cdot c_j
=& \sum_{p=i}^j (-1)^{j-p}\binom{k+j-1}{k+p-1}\binom{k+p-1}{k+i-1} \\
=& \sum_{p=i}^j (-1)^{j-p}\frac{(k+j-1)!(k+p-1)!(j-i)!}{(k+i-1)!(j-i)!(k+p-1)!(j-p)!(p-i)!} \\
=& \binom{k+j-1}{k+i-1} \sum_{p=i}^j \binom{j-i}{j-p}(-1)^{j-p} \\
=& \binom{k+j-1}{k+i-1} (1-1)^{j-i}=0 .
\end{align*}
\end{proof}

The rest of this section is devoted to the proof of the Main formula I.
\subsection{Proof of the Main Theorem}
For each coefficient $\gamma_l=\gamma_l(m,n,k)$ in~\eqref{eq: gamma}, we have:
$\gamma_l = \int_{\bbP^N} H^{N-l}\cap c_M(\tau_{m,n,k})$.

Since $\int_X \alpha =\int_Y f_*\alpha$ for any class $\alpha$ and any proper morphism 
$f\colon X\to Y$, by the projection formula we have
\begin{align*}
\gamma_l
=&~ \int_{\bbP^N} H^{N-l} \cap c_M(\tau_{m,n,k})
=~ \int_{\tau_{k}} c_1(\calO(1))^{N-l}\cap \nu_* c_{SM}(\hat\tau_{m,n,k}) \\
=&~ \int_{\hat\tau_{m,n,k}} c_1(\calO(1))^{N-l}\cap c_{SM}(\hat\tau_{m,n,k}) 
=~ \int_{\hat\tau_{m,n,k}} c_1(\calO(1))^{N-l}c(\calT_{\hat\tau_{m,n,k}})\cap [\hat\tau_{m,n,k}] 
\end{align*}
where we denote by
$\calO(1)$ the pull back of the hyperplane bundle on $\bbP^N$ by $i$ and $\nu$.
Note that the pull-back of this bundle to $\hat\tau_{m,n,k}$ agrees with the
tautological line bundle $\calO_{\hat\tau_{m,n,k}}(1)$.

By the Euler sequence~\eqref{eulersequence} and
the Whitney formula one has
\begin{equation*}
c(\calT_{\hat\tau_{m,n,k}}) = c(\rho^*(Q^{\vee m})\otimes \calO_{\hat\tau_{m,n,k}}(1))c(\rho^*\calT_{G(k,n)}).
\end{equation*}
So we get.
\begin{equation*}
\gamma_l = \int_{\hat\tau_{m,n,k}} c(\rho^*\calT_{G(k,n)})c(\rho^*(Q^{\vee m})\otimes \calO(1))c_1(\calO(1))^{N-l}\cap [\hat\tau_{m,n,k}] .
\end{equation*}

This expression may be expanded using \cite[Example 3.2.2]{INT},
and we obtain
\label{lemma; gamma}
\begin{lemma}
{\small
\begin{align*}
\gamma_l
=&~ \int_{\hat\tau_{m,n,k}} \sum_{p=0}^{m(n-k)}\sum_{i=0}^{p} \binom{m(n-k)-i}{p-i}
c(\rho^*\calT_{G(k,n)})c_i(\rho^*(Q^{\vee m}))c_1(\calO(1))^{N-l+p-i} \cap [\hat\tau_{m,n,k}]
\end{align*}
}
\end{lemma}
We want to express the coefficients $\gamma_l$ as degrees of certain
classes in $G(k,n)$. For this, we will need the following observation.
\begin{lemma}
Let $S$ denote the universal subbundle of the Grassmannian. Then
\[
\rho_*(c_1(\calO(1))^e\cap [\hat\tau_{m,n,k}])= c_{e+1-m(n-k)}({S^\vee}^m)\cap [G(k,n)].
\] 
\end{lemma}
\begin{proof}
Since $\hat\tau_{m,n,k}=\bbP({Q^\vee}^m)$, by the definition of Segre class~\cite[\S 3.1]{INT}, we have:
\[
\rho_*(c_1(\calO(1))^e\cap [\hat\tau_{m,n,k}])=s_{e+1-m(n-k)}({Q^\vee}^m)\cap [G(k,n)].
\]

The exact sequence over the Grassmanian
\[
\begin{tikzcd}
0 \arrow{r} & S \arrow{r} & \calO_{G(k,n)}^n \arrow{r} & Q \arrow{r} & 0
\end{tikzcd}
\]
implies that  $s(Q)=c(Q)^{-1}=c(S)$. Taking duals we get the desired formula.   
\end{proof} 

Since $N=mn-1$, we have $N+1-m(n-k)=mk$, thus
\begin{align*}
\rho_*(c_1(\calO(1))^{N}\cap [\hat\tau_{m,n,k}])
=&~ c_{N+1-m(n-k)}({S^\vee}^m)\cap [G(k,n)] \\
=&~ c_{mk}({S^\vee}^m)\cap [G(k,n)].
\end{align*}
Pushing forward to $G(k,n)$ one gets 
{\small
\begin{align*}
\gamma_l
=&~ \int_{\hat\tau_{m,n,k}} \sum_{p=0}^{m(n-k)}\sum_{i=0}^{p} \binom{m(n-k)-i}{p-i} 
c(\rho^*\calT_{G(k,n)})c_i(\rho^*(Q^{\vee m}))c_1(\calO(1))^{N-l+p-i} \cap [\hat\tau_{m,n,k}] \\
=&~ \int_{G(k,n)} \sum_{p=0}^{m(n-k)}\sum_{i=0}^{p} \binom{m(n-k)-i}{p-i} 
c(\calT_{G(k,n)})c_i(Q^{\vee m})c_{mk-l+p-i}({S^\vee}^m) \cap [G(k,n)]\\
=&~ \sum_{p=0}^{m(n-k)}\sum_{i=0}^{p} \binom{m(n-k)-i}{p-i} 
\int_{G(k,n)} c(\calT_{G(k,n)})c_i(Q^{\vee m})c_{mk-l+p-i}({S^\vee}^m) \cap [G(k,n)]\\
=&~ \sum_{p=0}^{m(n-k)}\sum_{i=0}^{p} \binom{m(n-k)-i}{p-i} 
\int_{G(k,n)} c(\calT_{G(k,n)})c_i(Q^{\vee m})c_{mk-l+p-i}({S^\vee}^m) \cap [G(k,n)].
\end{align*}
}

Recall that $\calT_{G(k,n)}=Hom(S,Q)=S^\vee\otimes Q $. Hence the formula reads as:
\[
\gamma_l
= \sum_{p=0}^{m(n-k)}\sum_{i=0}^{p} \binom{m(n-k)-i}{p-i} 
\int_{G(k,n)} c(S^\vee\otimes Q)c_i(Q^{\vee m})c_{mk-l+p-i}({S^\vee}^m) \cap [G(k,n)].
\]

Define the integration part by:
\[
\alpha_{i,p}^l(k):=\int_{G(k,n)} 
c(S^\vee\otimes Q)c_i(Q^{\vee m})c_{mk-l+p-i}({S^\vee}^m) \cap [G(k,n)] 
\]
and recall that
\[
B_{i,p}(k) =\binom{m(n-k)-p}{i-p}, \quad B(k)=\left[ B_{i,p}\right]_{i,p}.
\]
Let $\alpha^l(k)=\left[ \alpha_{i,p}^l\right]_{i,p}$
be the $m(n-k)-1\times m(n-k)-1$ matrix, we have:
\label{formula; gamma}
\begin{align*}
\gamma_l
=&~ \sum_{p=0}^{m(n-k)}\sum_{i=0}^{p} \binom{m(n-k)-i}{p-i} \alpha_{i,p}^l(k) \\
=&~ \sum_{p=0}^{m(n-k)}\sum_{i=0}^{p}  \alpha_{i,p}^l(k) B_{p,i}(k) \\
=&~ \sum_{p=0}^{m(n-k)}\sum_{i=0}^{m(n-k)}  \alpha_{i,p}^l(k) B_{p,i}(k) \\
=&~ \tr (\alpha^l(k)\cdot B(k)).
\end{align*}

\begin{remark} \
\begin{enumerate}
\item The last step is based on the convention that $\binom{x}{y}=0$ for $y<0$.
\item  $\alpha_{i,p}^l(k)$ and $B_{i,p}(k)$ are functions of $m,n,k$, and so are the matrices $\alpha^l(k)=\alpha^l(m,nk)$ and $B(k)=B(m,n,k)$.
\item Note that when $l=mk$, 
\[
\alpha_{i,p}^{mk}(k):=\int_{G(k,n)} 
c(S^\vee\otimes Q)c_i(Q^{\vee m})c_{p-i}({S^\vee}^m) \cap [G(k,n)]
\]
equals the coefficient $A_{i,p}(k)$ introduced in the theorem. Hence we have $A(k)=\alpha^{mk}(k)$.
\end{enumerate}
\end{remark}
Now we are ready to complete the proof of Theorem~\ref{thm; cmalgorithm}. Recall that
$\calH(k)=\calH(m,n,k)= [H^{mk-i+j}]_{ij}$.
\begin{prop}
We have the following formula: 
\[
\Gamma_k(H)= \nu_*c_{SM}(\hat\tau_{m,n,k})=\tr(\alpha^{mk}(k)\cdot \calH(k)\cdot B(k))=\tr(A(k)\cdot \calH(k)\cdot B(k)).
\]
\end{prop} 
\begin{proof}
First we multiply the matrix $\alpha^{mk}(k)$ and $\calH(k)$, for any $i=0,1\cdots m(n-k)$, the $i$th row is
\[
\left[ 
\sum_{p=0}^{m(n-k)} \alpha^{mk}_{i,p}(k)\cdot H^{mk-p}; \sum_{p=0}^{m(n-k)} \alpha^{mk}_{i,p}(k)\cdot H^{mk+1-p};\cdots ; \sum_{p=0}^{m(n-k)} \alpha^{mk}_{i,p}(k); \cdot H^{mn-p}
\right].
\]
Multiplying with the $i$th column of $B(k)$ gives the $i$th row and $i$th column entry of $\alpha^{mk}(k)\cdot \calH(k)\cdot B(k)$:
\[
\sum_{j=0}^{m(n-k)}\sum_{p=0}^{m(n-k)} \alpha^{mk}_{i,p}(k)H^{mk+j-p}B_{j,i}(k) .
\]
Hence:
\begin{align*}
\numberthis\label{eq; cmcoe}
\tr(A(k)\cdot \calH(k)\cdot B(k))
=&~ \tr(\alpha^{mk}(k)\cdot \calH(k)\cdot B(k)) \\
=&~ \sum_{i=0}^{m(n-k)}\sum_{j=0}^{m(n-k)}\sum_{p=0}^{m(n-k)} \alpha^{mk}_{i,p}(k)H^{mk+j-p} B_{j,i}(k) \\
=&~ \sum_{i,j,p=0}^{m(n-k)} \alpha^{mk}_{i,p}(k)B_{j,i}(k)H^{mk+j-p}
=~ \sum_l \gamma'_l H^l . 
\end{align*}
Here 
\[
\gamma'_l=\sum_i\sum_{mk+j-p=l} \alpha^{mk}_{i,p}(k)B_{j,i}(k).
\]
Now we fix the power of $H$, and compare the coefficients $\gamma'_l$ to the values of $\gamma_l$  in Lemma~\ref{formula; gamma}. The proposition holds if we can show that they match. In order to compare them, we will need the following lemma:
\begin{lemma} \
\label{lemma; Aip}
For all $j$, $\alpha^l_{i,p}(k)=\alpha^{l+j}_{i,p+j}(k)$. Moreover, if $l+i-p>mk$ or $mk-l+p>m(n-k)$,
then $\alpha^l_{i,p}(k)=0$ .
\end{lemma}
\begin{proof}
The first assertion is an immediate consequence of
the definition of $\alpha^l_{i,p}(k)$. For the second part, if $l+i-p>mk$, then $mk-l+p-i<0$ and $c_{mk-l+p-i}({S^\vee}^m)=0$.
If $mk-l+p>m(n-k)$, we can treat
$
c_i(Q^{\vee m})c_{mk-l+p-i}({S^\vee}^m)
$
as a group homomorphism from $A_*(G)$ to $A_{*-(i+mk-1+p-i)}(G)=A_{*-(mk-1+p)}(G)$, where $G=G(k,n)$. 
It vanishes because $\dim G=k(n-k)\leq n(n-k)\leq m(n-k)< (mk-1+p)$.
\end{proof}

Depending on the values of $l=mk+j-p$ and $mk$, we separate into 
two cases:
\begin{enumerate} 
\item Case 1: $j-p=s\geq 0$, $l\geq mk$. \\
We let $j=p+s$, then $\gamma'_{mk+s}$ can be expressed as:
\begin{align*}
\gamma'_{mk+s}
=&~ \sum_{i=0}^{m(n-k)}\sum_{p=0}^{m(n-k)-s} \alpha^{mk}_{i,p}(k)B_{j,i}(k) \\
=&~ \sum_{i=0}^{m(n-k)}\sum_{j=s}^{m(n-k)} \alpha^{mk}_{i,p}(k)B_{j,i}(k) \\
=&~ \sum_{i=0}^{m(n-k)}\sum_{j=s}^{m(n-k)} \alpha^{mk}_{i,j-s}(k)B_{j,i}(k)
\end{align*}
By Lemma~\ref{lemma; Aip} we have:
\begin{align*}
\gamma'_{mk+s}
= \sum_{i=0}^{m(n-k)}\sum_{j=s}^{m(n-k)} \alpha^{mk+s}_{i,j}(k)B_{j,i}(k). 
\end{align*}
and $\alpha^{mk+s}_{i,j}(k)=0$ when $j<s$, for $mk+s+i-j>mk$. Hence
\begin{align*}
\gamma'_{mk+s}
=&~ \sum_{i=0}^{m(n-k)}\sum_{j=0}^{m(n-k)} \alpha^{mk+s}_{i,j}B_{j,i}(k) \\
=&~ \tr(\alpha^{mk+s}(k)B(k)) \\
=&~ \gamma_{mk+s}.
\end{align*}

\item Case 2: $p-j=s>0$, $l<mk$. \\
We let $p=j+s$, and by Lemma~\ref{lemma; Aip} we have.
\begin{align*}
\gamma'_{mk-s}
=&~ \sum_{i=0}^{m(n-k)}\sum_{p=s}^{m(n-k)} \alpha^{mk}_{i,p}(k)B_{j,i}(k) \\
=&~ \sum_{i=0}^{m(n-k)}\sum_{j=0}^{m(n-k)-s} \alpha^{mk}_{i,j+s}(k)B_{j,i}(k) \\
=&~ \sum_{i=0}^{m(n-k)}\sum_{j=0}^{m(n-k)-s} \alpha^{mk-s}_{i,j}(k)B_{j,i}(k).
\end{align*}
Again, notice here that if $j>m(n-k)-s$, then $mk-(mk-s)+j=s+j>m(n-k)$, 
and Lemma~\ref{lemma; Aip} implies that $\alpha^{mk-s}_{i,j}(k)=0$ in this case. Hence
\begin{align*}
\gamma'_{mk-s}
=&~ \sum_{i=0}^{m(n-k)}\sum_{j=0}^{m(n-k)} \alpha^{mk-s}_{i,j}(k)B_{j,i}(k) \\
=&~ \tr(\alpha^{mk-s}(k)B(k)) \\
=&~ \gamma_{mk-s}.  
\end{align*}
\end{enumerate} 
This shows that $\gamma'_l=\gamma_l$ for all $l$ from $0$ to $N$, hence concludes the proof of the proposition, which leads to Theorem~\ref{thm; cmalgorithm}.
\end{proof}
\begin{remark} 
In the expression of $\calH(k)$, it seems that we are also involving the negative powers of $H$. 
In fact, however, for $l<0$, if $p\geq i$, we have $mk-l+p-i>mk$, then Lemma~\ref{lemma; Aip} shows that $A_{i,p}^l(k)=0$; if $p<i$, we have $B_{p,i}(k)=0$. In either case
\[
\gamma_l=\gamma'_l=\sum_{mk+j-p=l} \alpha^{mk}_{i,p}(k)B_{j,i}(k)=0 .
\]
Hence the negative powers of $H$ do not appear the expression.
\end{remark}

\section{Projectivized Conormal Cycle and Characteristic Cycle of $\tau_{m,n,k}$.}
\label{S; projconormal}
In \cite{Kennedy}, Kennedy generalized MacPherson's theorem (Theorem~\ref{thm; macc_*}) from $\bbC$ to arbitrary algebraically closed field of characteristic $0$. The key ingredients in the paper are the conormal and Lagrangian cycles. Here we review some basic definitions.
Let $X\subset M$ be a $d$-dimensional subvariety of $m$-dimensional smooth ambient space. The \textit{conormal space} of $X$ is defined as a dimension $m$ subvariety of $T^*M$
\[
T^*_X M:=\overline{\{(x,\lambda)|x\in X_{sm};\lambda(T_xX)=0\}}\subset T^*M
\]
We call the class $[T^*_X M]$ the \textit{Conormal cycle of $X$}.
Let $L(M)$ be the free abelian group generated by the conormal spaces $T_V^*M$ for subvarieties $V\subset M$. As shown in \cite{Kennedy}, the group $L(M)$ is isomorphic to the group of constructible functions $F(M)$ by mapping $Eu_V$ to $(-1)^{\dim V} T^*_V M$. 
We call an element of $L(M)$ a \textit{(conical) Lagrangian cycle} of $M$. We say a Lagrangian cycle is \textit{irreducible} if it equals the conormal space of some subvariety $V$. The name Lagrangian comes from \cite[Lemma 3]{Kennedy}, which says that the conormal spaces of closed subvarieties $V\subset M$ are exactly the conical Lagrangian varieties of $T^*M$. 

We define the \textit{projectivized conormal cycle} of $X$ to be $Con(X):=[\bbP(T^*_X M)]$, which is a $m-1$-dimensional cycle in the total space $\bbP(T^*M)$ . There is a group morphism $ch\colon F(M)\to A_{m-1}(\bbP(T^*M))$ sending $Eu_V$ to $(-1)^{\dim V} Con(V)$. The cycle $ch(\ind_X)$ is called the \textit{`Characteristic Cycle of $X$'}, and denoted by $Ch(X)$. As pointed in \cite[\S 4]{MR2097164}, the classes $ch(\varphi)$ and $c_*(\varphi)$ are related by `casting the shadow ' process, which we now explain. 

Let $E\to M$ be a rank $e+1$ vector bundle. For a projective bundle $p\colon \bbP(E)\to M$, the structure theorem \cite[\S 3.3]{INT} shows that any $C\in A_r(\bbP(E))$ can be uniquely written as $C=\sum_{i=0}^{e} c_1(\calO(1))^i\cap p^*(C_{r+i-e})$. 
Here $\calO(1)$ is the tautological line bundle on $\bbP(E)$, and $C_{r+i-e}\in A_{r+i-e}(M)$ are classes in $M$. 
We define the class $C_{r-e}+C_{r-e+1}+\cdots + C_r \in A_*(M)$ to be the \textit{Shadow} of $C$. 

Also, for any constructible function $\varphi\in F(M)$, we defined the signed class $\breve{c}_*(\varphi)\in A_*(M)$ to be
$\{\breve{c}_*(\varphi)\}_r = (-1)^r\{c_*(\varphi)\}_r $. 
Here for any class $C\in A_*(M)$, $C_r$ denotes the $r$-dimensional piece of $C$. 

The following lemma (\cite[Lemma 4.3]{MR2097164}) shows that the Lagrangian cycle $ch(\varphi)$ and  the class $c_*(\varphi)$ are essentially equivalent. 
\begin{lemma}
The class $\breve{c}_*(\varphi)$ is the shadow of the characteristic cycle $ch(\varphi)$.
\end{lemma}

In particular, when $M=\bbP^N$ we have the following diagram
\[
\begin{tikzcd}
 \bbP(T^* \bbP^N) \arrow{r}{j} \arrow{d}{\pi} & \bbP^{N}\times\bbP^{N} \arrow{dl}{pr_1}\arrow{d}{pr_2} \\
M=\bbP^{N} & \bbP^{N}
\end{tikzcd} .
\]
Let $L_1,L_2$ are the pull backs of the line bundle $\calO(1)$ of $\bbP^N$ from projections $pr_1$ and $pr_2$. Then we have 
$\calO_P(1)=j^*(L_1\otimes L_2)$, and $j_*[\bbP(T^*_X)]=c_1(L_1\otimes L_2) \cap [\bbP^{N}\times\bbP^{N}]$ is a divisor in $\bbP^{N}\times\bbP^{N}$. 

For $i=1,2$, let $h_i=c_1(L_i)\cap [\bbP^{N}\times\bbP^{N}] $ be the pull backs of hyperplane classes. 
\begin{prop}
\label{prop; chc_*}
Let $\varphi$ be a constructible function on $\bbP^N$. Write $c_*(\varphi)=\sum_{l=0}^{N} \gamma_lH^{N-l}$ as a polynomial of $H$ in $A_*(\bbP^N)$. Then we have
\begin{align*}
ch(\varphi)=\sum_{j=1}^{N}\sum_{k=j-1}^{N-1} (-1)^k\gamma_k\binom{k+1}{j} h_1^{N+1-j}h_2^j \/.
\end{align*}
as a class in $\bbP^{N}\times\bbP^{N}$.
\end{prop}
\begin{proof}
The lemma shows that $\breve{c}_*(\varphi)=\sum_{l=0}^{N} (-1)^l\gamma_l H^{N-l}$ is the shadow of $ch(\varphi)$. 
Then the structure theorem of the projective bundle $\bbP(T^* \bbP^N)$ shows that
\begin{align*}
ch(\varphi)=\sum_{k=0}^{N-1} (-1)^k\gamma_k c_1(\calO(1))^k\cap \pi^*([\bbP^k]) \/.
\end{align*}
Hence 
\begin{align*}
j_*ch(\varphi)
&= \sum_{k=0}^{N-1} (-1)^k\gamma_k j_*(c_1(\calO(1))^k\cap \pi^*([\bbP^k])) \\
&=\sum_{k=0}^{N-1} (-1)^k\gamma_k j_*(c_1(\calO(1))^k\cap \pi^*(H^{N-k}\cap [\bbP^N])) \\
&=\sum_{k=0}^{N-1} (-1)^k\gamma_k c_1(L_1\otimes L_2)^k\cap j_*\pi^*(H^{N-k}\cap [\bbP^N]) \\
&=\sum_{k=0}^{N-1} (-1)^k\gamma_k c_1(L_1\otimes L_2)^k c_1(L_1)^{N-k} \cap j_*[\bbP(T^*\bbP^N)]\\
&=\sum_{k=0}^{N-1} (-1)^k\gamma_k c_1(L_1\otimes L_2)^{k+1} c_1(L_1)^{N-k} \cap [\bbP^{N}\times\bbP^{N}]  \\
&=\sum_{k=0}^{N-1} (-1)^k\gamma_k (h_1+h_2)^{k+1} h_1^{N-k}  \\
&=\sum_{j=1}^{N}\sum_{k=j-1}^{N-1} (-1)^k\gamma_k\binom{k+1}{j} h_1^{N+1-j}h_2^j. \qedhere
\end{align*}
\end{proof}
 
Let $X\subset \bbP^N$ be a $d$-dimensional projective variety. We define the $k$-th polar class of $X$ to be:
\[
[M_k]:=\overline{\{x\in X_{sm} | \dim (T_x X_{sm}\cap L_k)\geq k-1  \}} \/.
\]
Here $L_k\subset L^N$ is a linear subspace of codimension $d-k+2$. This rational equivalence class is independent of $L_k$ for general $L_k$. The polar classes are closely connected to the Chern-Mather class of $X$, as pointed out in \cite{Piene}:
\begin{thm}[Theorem 3~\cite{Piene}]
The $k$-th polar class of $X$ is given by
\[
[M_k]=\sum_{i=0}^k (-1)^i \binom{d-i+1}{d-k+1}  H^{k-i}\cap c^i_M(X).
\]
Here $c_M^i(X)$ is the codimension $i$ piece of the Chern-Mather class of $X$.
\end{thm}
As recalled in \cite[Remark 2.7]{Aluffi}, the degrees of the polar classes are exactly the coefficients appearing in the projectivized conormal cycle. 
Let $X\subset \bbP^N$ be a $d$-dimensional subvariety. Let $\delta_l$ be the degree of the $l$th polar class $M_l$. Then the projectivized conormal cycle of $X$ is given by
\[
Con(X)=\sum_{l=0}^{d} \delta_{N-d+l} h_1^l h_2^{N+1-l}
\]
as a class in $\bbP^N\times \bbP^N$.
Hence the coefficients of the projectivized conormal cycle of $X$ (as a class in $\bbP^N\times \bbP^N$) are exactly the polar degrees of $X$.

For the determinantal variety $\tau_{m,n,k}$, the above discussion yields the
following.
\begin{prop}
\label{prop; contau}
Let $c_M(\tau_{m,n,k})=\sum_{l=0}^{mn-1} \beta_l H^{mn-1-l}$ be the Chern-Mather class of $\tau_{m,n,k}$ in $A_*(\bbP^{mn-1})$, as obtained in \eqref{eq; betacm}. Then the projectivized conormal cycle $Con(\tau_{m,n,k})$ equals:
\[
Con(\tau_{m,n,k})=  (-1)^{(m+k)(n-k)-1}\sum_{j=1}^{mn-2} \sum_{l=j-1}^{mn-2} (-1)^l\beta_l\binom{l+1}{j} h_1^{mn-j}h_2^{j}\cap [\bbP^{mn-1} \times \bbP^{mn-1}] \/.
\]
Via Corollary~\ref{thm; csmalgorithm}, the characteristic cycle of $\tau_{m,n,k}$ and $\tau_{m,n,k}^\circ$ are given by
\[
Ch(\tau_{m,n,k})=\sum_{i=0}^{n-1-k}(-1)^{i}\binom{k+i-1}{k-1} Con(\tau_{m,n,k+i}) ;
\]
and 
\[
Ch(\tau_{m,n,k}^\circ)=\sum_{i=0}^{n-1-k}(-1)^{i}\binom{k+i}{k} Con(\tau_{m,n,k+i}) \/.
\]
\end{prop}

For a projective variety $X$, the sum of its polar degrees is also a very interesting invariant. It is called the generic Euclidean distance degree of $X$, and denoted by $gED(X)$. We refer to \cite{MR3451425} for more details about Euclidean distance degree. For determinantal varieties, we have
\begin{prop}
\label{prop; ged}
The generic Euclidean distance degree of $\tau_{m,n,k}$ is given by
\[
gED(\tau_{m,n,k})=\sum_{l=0}^{(m+k)(n-k)-1}\sum_{i=0}^{l} (-1)^{i}\binom{(m+k)(n-k)-i}{(m+k)(n-k)-l}\beta_{(m+k)(n-k)-1-i}.
\]
where $\sum_{l=0}^{N} \beta_l[\bbP^l]$ is the Chern-Mather class of $\tau_{m,n,k}$ as obtained in \eqref{eq; betacm}.
\end{prop}

For example, we list the generic Euclidean distance degree for some of the determinantal varieties here. 

\begin{tabular}{c | *{8}{c}}
\hline
  $n$  & $2$  & $3$ & $4$ & $5$ & $6$ & $7$ & $8$ & $9$   \\
\hline
$gED(\tau_{n,n,1})$ & 6 & 39 & 284 & 2205 & 17730 & 145635 & 1213560 & 10218105 \\
$gED(\tau_{n+1,n,1})$ & 10 & 83 & 676 & 5557 & 46222 & 388327 & 3288712 &28031657 \\
$gED(\tau_{n,n,n-1})$ & 6 & 39 & 284 & 2205 & 17730 & 145635 & 1213560 & 10218105 \\
\hline 
\end{tabular}

Notice that the first and the last sequence are the same, this is because that $\tau_{n,n,1}$ and $\tau_{n,n,n-1}$ are dual varieties, and their polar degrees are `flipped'. We will explain details in \S\ref{S; dual}. 

Moreover, the third sequence matches the sequence in \cite[Table 1]{Duco}, as the number of nonlinear normal modes for a fully resonant Hamiltonian system with $n$ degrees of freedom. The proof will be presented elsewhere, based on an explicit computation of the Euler characteristic of certain Hamiltonian hypersurfaces in $\bbP^n\times \bbP^n$.

\section{The Characteristic Cycle of the intersection cohomology sheaf of $\tau_{m,n,k}$}
\label{S; char}
For a stratified smooth complex variety $M$,
we can also assign Lagrangian cycles to constructible sheaves on it, in particular the intersection cohomology sheaf of a subvariety $V\subset M$. In this section we assume our base field to be $\bbC$, and prove that for the determinantal variety $\tau_{m,n,k}\subset \bbP^{mn-1}$, the Lagrangian cycle assigned to its intersection cohomology sheaf is irreducible, or equivalently, the microlocal multiplicities are all $0$ except for the top dimension piece.

\subsection{Characteristic Cycle of a Constructible Sheaf}
Let $M$ be a smooth  compact complex algebraic variety, and $\sqcup_{i\in I} S_i$ be a Whitney stratification of $M$. For any constructible sheaf $\calF^\bullet$ with respect to the stratification, one can assign a cycle in the cotangent bundle $T^*M$ to $\calF^\bullet$. This cycle is called the \textit{Characteristic Cycle} of $\calF^\bullet$, and is denoted by $\calCC(\calF^\bullet)$. The cycle can be expressed as a Lagrangian cycle
\[
\calCC(\calF^\bullet)=\sum_{i\in I} r_i(\calF^\bullet)[T^*_{\overline{S_i}}M] \/.
\]
Here the integer coefficients $r_i(\calF^\bullet)$ are called the \textit{Microlocal Multiplicities}, and are explicitly constructed in \cite[\S 4.1]{Dimca} using the $k$th Euler obstruction of pairs of strata and the stalk Euler characteristic $\chi_i(\calF^\bullet)$. For any $x\in S_i$, the stalk Euler characteristic $\chi_i(\calF^\bullet)$ are defined as:
\[
\chi_i(\calF^\bullet)=\sum_{p} (-1)^p\dim H^p(\calF^\bullet(x)).
\]

For the stratification  $\sqcup_{j\in I} S_j$ of $M$, we define 
\[
e(j,i)=Eu_{\overline{S_i}}(S_j):=Eu_{\overline{S_i}}(x); x\in S_j
\]
to be the local Euler obstruction along the $j$th stratum $S_j$ in the closure of $S_i$. If $S_j\not\subset \overline{S_i}$, then we define $e(j,i)=0$.

The following deep theorem from \cite[Theorem 3]{Du}, \cite[Theorem 6.3.1]{Ka} reveals the relation the microlocal multiplicities $r_i(\calF^\bullet)$, the stalk Euler characterictic $\chi_i(\calF^\bullet)$, and the local Euler obstructions $e(i,j)$ :
\begin{thm}[Microlocal Index Formula]
\label{thm; microlocal}
For any $j\in I$, we have the following formula:
\[
\chi_j(\calF^\bullet)=\sum_{i\in I}e(j,i)r_i(\calF^\bullet)  \/.
\]
\end{thm}

This theorem suggests that if one knows about any two sets of the indexes, then one can compute the third one. 

\subsection{Intersection Cohomology Sheaf}\label{S; ICS}
For any subvariety $X\subset M$, Goresky and MacPherson defined a sheaf of bounded complexes $\calIC$ on $M$ in \cite{G-M2}. This sheaf is determined by the structure of $X$, and is usually called the \textit{Intersection Cohomology Sheaf} of $X$. For details about the intersection cohomology sheaf and intersection homology one is refereed to \cite{G-M1} and \cite{G-M2}. This sheaf is constructible with respect to any Whitney stratification on $M$, hence one can define its characteristic cycle $\calCC(\calIC)$.
We define $\calCC(\calIC)$ to be the \textit{Characteristic Cycle of the Intersection Cohomology Sheaf of $X$}, and we will call it the $IC$ characteristic cycle for short.

When $X$ admits a small resolution, the intersection cohomology sheaf $\calIC$ is derived from the constant sheaf on the small resolution.
\begin{thm}[Goresky, MacPherson {{\cite[\S 6.2]{G-M2}}}]
\label{thm; stalk}
Let $X$ be a $d$-dimensional irreducible complex algebraic variety. Let $p\colon Y\to X$ be a small resolution of $X$, then 
\[
\calIC\cong Rp_*\bbC_Y[2d] .
\]
In particular, for any point $x\in X$ the stalk Euler characteristic of $\calIC$ equals the Euler characteristic of the fiber:
\[
\chi_x (\calIC)=\chi (p^{-1}(x)).
\]
\end{thm}

\subsection{Characteristic Cycle of the Intersection Cohomology Sheaf of $\tau_{m,n,k}$} 
\label{S; charcycle}
In terms of the determinantal varieties $\tau_{m,n,k}$, Proposition~\ref{prop; small} shows that $\nu\colon \hat\tau_{m,n,k}\to \tau_{m,n,k}$ is a small resolution. Let $X:=\tau_{m,n,k}\subset \bbP^{mn-1}:=M$ be the embedding. Since $\bbP^{mn-1}=\tau_{m,n,0}$, we consider the Whitney stratification $\{S_i:=\tau_{m,n,i}^\circ| i=0,1,\cdots ,n-1\}$ of $\bbP^{mn-1}$. One has the following observations:
\begin{enumerate}
\item $\overline{S_i}=\tau_{m,n,i}$ ; 
\item $X=\tau_{m,n,k}=\overline{S_k}$ ; 
\item $S_i\subset \overline{S_j}$ if and only if $i\geq j$. 
\end{enumerate}

The main result of this section is the following theorem:
\begin{thm}
\label{thm; charcycle}
For $i=0,1,\cdots ,n-1$, let $r_i:=r_i(\calICT)$ be the microlocal multiplicities. Then
\[
r_i = 
\begin{cases} 
0, & \mbox{if } i\neq k \\ 
1, & \mbox{if } i=k
\end{cases} \/.
\] 
Hence we have:
\[
\calCC(\calICT)=[T^*_{\tau_{m,n,k}}\bbP^{mn-1}]
\]
is irreducible.
\end{thm}
\begin{proof}
Let $\nu\colon \hat\tau_{m,n,k}\to \tau_{m,n,k}$ be the small resolution. For any $p\in S_j=\tau_{m,n,j}^\circ$, we have $\nu^{-1}(p)=G(k,j)$.
Theorem~\ref{thm; stalk} shows that for any $j=0,1,\cdots ,n-1$, any $p\in S_j$, the stalk Euler characteristic equals:
\[
\chi_j(\calICT)=\chi_{p}(\calICT)=\chi(\nu^{-1}(p))=\binom{j}{k} \/.
\]
And Theorem~\ref{thm; eulerobstruction} shows that:
\[
e(j,i)=Eu_{\overline{S_i}}(S_j)=Eu_{\tau_{m,n,i}}(\tau_{m.n.j}^\circ)=\binom{j}{i} \/.
\]
Here when $a<b$ we make the convention that $\binom{a}{b}=0$.

Hence the Microlocal Index Formula~\ref{thm; microlocal} becomes:
\[
\chi_j(\calICT)=\binom{j}{k}=\sum_{i\in I}e(j,i)r_i(\calICT)=\sum_{i\in I}\binom{j}{i}r_i \/.
\]
One gets the following linear equation:
\[
  \left[ {\begin{array}{c}
   \binom{0}{k} \\
   \binom{1}{k} \\
   \cdots       \\
   \binom{k}{k} \\
   \binom{k+1}{k} \\
   \cdots        \\
   \binom{n-1}{k} 
  \end{array} } 
  \right]
=
\left[ {\begin{array}{ccccccc}
   \binom{0}{0}    &          0   &          &           &         &    &    \\
   \binom{1}{0}    &\binom{1}{0}  &          &           &         &    &    \\
   \cdots          &     \cdots   &          &           &         &    &    \\
   \binom{k}{0}    &  \binom{k}{1}& \cdots   & \binom{k}{k} &         &    &  \\                
   \binom{k+1}{0}  & \binom{k+1}{1}   &   \cdots        &    \binom{k+1}{k} & \binom{k+1}{k+1}        &    &   \\
   \cdots          &      \cdots    &     \cdots      &     \cdots    &    \cdots    &  \cdots  & \\
   \binom{n-1}{0}  & \binom{n-1}{1}   &    \cdots      &   \cdots    & \cdots  &\cdots   & \binom{n-1}{n-1}\\
  \end{array} } 
  \right]   
\cdot
\left[ {\begin{array}{c}
   r_0 \\
   r_1 \\
   \cdots       \\
   r_k\\
   r_{k+1} \\
   \cdots        \\
   r_{n-1} 
  \end{array} } 
  \right]   \/.   
\]
Notice that the middle matrix is invertible, since the diagonals are all non-zero, hence the solution vector $[r_0,r_1,\cdots ,r_{n-1}]^T$ is unique. Hence 
\[
r_i = 
\begin{cases} 
0, & \mbox{if } i\neq k \\ 
1, & \mbox{if } i=k
\end{cases} 
\]  
is the unique solution to the system.
\end{proof}
This theorem shows that the projectivized conormal cycle $Con(\tau_{m,n,k})$ is exactly the $IC$ characteristic cycle $\calCC(\calICT)$ for the determinantal variety $\tau_{m,n,k}$.  As pointed out in the previous section, the Chern-Mather class $c_M(\tau_{m,n,k})$ is the `shadow' of  $Con(\tau_{m,n,k})$ . Hence our formula (Prop~\ref{prop; contau}) explicitly computes the $IC$ characteristic cycle $\calCC(\calICT)$ of determinantal varieties $\tau_{m,n,k}$. 

\begin{remark}
When the base field is $\bbC$, the lemma~\ref{lemma; cmeu} we proved in last section can actually be deduced from the following theorem.
\begin{thm*}[Theorem 3.3.1 \cite{Jones}]
Let $X$ be a subvariety of a smooth space $M$. Let $p\colon Y\to X$ be a small resolution of $X$. Assume that $\calCC(\calIC)$ is irreducible, i.e., 
$\calCC(\calIC)=[\overline{T_{X_{sm}}^*M}]$ is the conormal cycle of $\tau_{m,n,k}$.
Then 
\[
c_M(X)=p_*(c(T_Y)\cap [Y]) \/.
\]
\end{thm*}
As we have seen, the Tjurina transform $\hat\tau_{m,n,k}$ is a small resolution of $\tau_{m,n,k}$, and the $IC$ characteristic cycle of  $\tau_{m,n,k}$ is indeed irreducible. 

This theorem shows that under this assumption we can use any small resolution and its tangent bundle to compute the Chern-Mather class. However, as pointed out in \cite[Rmk 3.2.2]{Jones}, this is a rather unusual phenomenon. It is known to be true for Schubert varieties in a Grassmannian, for certain Schubert varieties in flag manifolds of types B, C, and D, and for theta divisors of Jacobians (cf. \cite{MR1642745}). 
By the previous section, determinantal varieties also share this 
property.
\end{remark}

\section{Examples and Observations}
\label{example}
Let $K$ be an algebraically closed field of characteristic $0$. In this section we will present some examples for determinantal varieties $\tau_{m,n,k}$ and $\tau_{m,n,k}^{\circ}$ over $K$.

They are computed by applying Theorem \ref{thm; cmalgorithm} and Corollary \ref{thm; csmalgorithm}, using the package
Schubert2 in Macaulay2~\cite{M2}. In the tables that follow, we give the coefficients of $[\bbP^i]$ in the polynomial expression of $c_{SM}(\tau_{m,n,k})$ in $A_*(\bbP^{mn-1})$.

\subsection{Chern-Mather classes}
Here we present some examples of $c_M(\tau_{m,n,k})$ for small $m$ and $n$. 
\begin{enumerate}
\item $\tau_{3,3,s}$ \\
\begin{tabular}{c | ccccccccc}
  Table  & $\bbP^{0}$ & $\bbP^{1}$  & $\bbP^{2}$  & $\bbP^{3}$  & $\bbP^{4} $  & $\bbP^{5} $ & $\bbP^{6}$  & $\bbP^{7}$ & $\bbP^{8}$ \\
\hline
$c_M(\tau_{3,3,0})$ & 9  & 36 & 84  & 126 & 126  & 84 & 36 & 9 & 1 \\
$c_M(\tau_{3,3,1})$ & 18 & 54 & 102 & 126 & 102  & 54 & 18 & 3 & 0 \\
$c_M(\tau_{3,3,2})$ & 9  & 18 & 24  & 18  & 6    &  0 &  0 & 0 & 0\\
\end{tabular}

\item $\tau_{4,3,s}$. \\
\begin{small}
\begin{tabular}{c | *{12}{c}}
  Table  & $\bbP^0$ & $\bbP^1$ & $\bbP^2$ & $\bbP^3$ & $\bbP^4$ & $\bbP^5$ & $\bbP^6$ & $\bbP^7$ &$\bbP^8$ & $\bbP^9$ & $\bbP^{10}$ & $\bbP^{11}$  \\
\hline
$c_M(\tau_{4,3,0})$       & 12  & 66 & 220 & 495 & 792 & 924  & 792 & 495 & 220& 66 & 12& 1 \\
$c_M(\tau_{4,3,1})$       & 24  & 96 & 248 & 444 & 564 & 514  & 336 & 153 & 44 & 6  & 0 & 0  \\
$c_M(\tau_{4,3,2})$       & 12  & 30 & 52  & 57  & 36  & 10   &  0  & 0   & 0  & 0  & 0 & 0  \\
\end{tabular}
\end{small}
\item $\tau_{4,4,s}$ \\
\begin{tabular}{c | *{8}{c}}
  Table  & $\bbP^{0}$ & $\bbP^{1}$  & $\bbP^{2}$  & $\bbP^{3}$  & $\bbP^{4} $  & $\bbP^{5} $ & $\bbP^{6}$  & $\bbP^{7}$  \\
\hline
$c_M(\tau_{4,4,1})$ &48 &288 &1128 &3168 &6672 & 10816 & 13716 & 13716    \\
$c_M(\tau_{4,4,2})$ &48 &216 &672 &1524 &2592 &3368 &3376  &2602     \\
$c_M(\tau_{4,4,3})$ &16 &48  &104 &152  &144  &80   & 20   & 0       \\
\end{tabular}
\vskip .05in
\begin{tabular}{c | *{8}{c}}
  Table  & $\bbP^{8}$  & $\bbP^{9}$ & $\bbP^{10}$ & $\bbP^{11}$ & $\bbP^{12}$ & $\bbP^{13}$ & $\bbP^{14}$ & $\bbP^{15}$   \\
\hline
$c_M(\tau_{4,4,1})$ &10816 &6672 &3168 &1128 &288 &48  & 4 & 0 \\
$c_M(\tau_{4,4,2})$ &1504 &616   &160  & 20  & 0  & 0  & 0 & 0  \\
$c_M(\tau_{4,4,3})$ & 0   & 0   & 0   & 0   & 0  & 0  & 0 & 0  \\
\end{tabular}
\end{enumerate}

These examples suggest that the coefficients appearing in the $c_M$ classes of
$\tau_{m,n,k}$ are non-negative. We conjecture that these
classes are effective (see Conjecture \ref{conjecture; effective}).

\subsection{Chern-Schwartz-MacPherson Class}
\subsubsection{Chern-Fulton class and Milnor Class}
Recall that $\tau_{n,n,1}$ is a degree $n$ hypersurface in $\bbP^{n^2-1}$ cut by the determinantal polynomial. Hence the normal bundle $\calN_{\tau_{n,n,1}}\bbP^{n^2-1}$ equals $\calO(n)|_{\tau_{n,n,1}}$, and 
\[
c_F(\tau_{n,n,1})= \frac{(1+H)^{n^2}\cdot nH}{1+nH}=nH(1+H)^{n^2}(\sum_i(-nH)^i)
\]
is the class of the virtual tangent bundle (cf.~\cite[Example 4.2.6(a)]{INT}) of $\tau_{n,n,1}$. This class is often called the
``Chern-Fulton class'', denoted by $c_F$. The Chern-Fulton class agrees with the
Chern-Schwartz-MacPherson class for nonsingular varieties, but differs from it
in general. The signed difference between the Chern-Fulton
class and the Chern-Schwartz-MacPherson class is called \textit{Milnor Class} 
(cf.~\cite{PP3}). 

\begin{ex}
When $m=n=3$, $k=1$, our algorithm computes the $c_{SM}$ class of $\tau_{3,3,1}$ as
\[
c_{SM}(\tau_{3,3,1}) = 3H+18H^2+54H^3+96H^4+108H^5+78H^6+36H^7+9H^8.
\]

Also, the Chern-Fulton class of $\tau_{3,3,1}$ equals
\begin{align*}
c_F(\tau_{3,3,1}) 
&= \frac{3H(1+H)^9}{1+3H} \\
&=3H+18H^2+54H^3+90H^4+108H^5+54H^6+90H^7-162H^8.
\end{align*}
Hence the Milnor class of $\tau_{3,3,1}$ is
\begin{align*}
\mathcal M(\tau_{3,3,1}) 
&= (-1)^7(c_F(\tau_{3,3,1})-c_{SM}(\tau_{3,3,1})) \\
&= 6H^4 + 24 H^6-54H^7+171H^8.
\end{align*}
This class is supported on $\tau_{3,3,2}$, but note that it is not a multiple of the $c_{SM}$
class of $\tau_{3,3,2}$.
\end{ex}

Now we show more examples about the Chern-Schwartz-MacPherson classes of $\tau_{m,n,k}$ and $\tau_{m,n,k}^\circ$
\subsubsection{$m=n=3$} 
\begin{small}
\begin{enumerate}
\item $\tau_{3,3,k}$ \\
\begin{tabular}{c | ccccccccc}
  Table  & $\bbP^{0}$ & $\bbP^{1}$  & $\bbP^{2}$  & $\bbP^{3}$  & $\bbP^{4} $  & $\bbP^{5} $ & $\bbP^{6}$  & $\bbP^{7}$ & $\bbP^{8}$ \\
\hline
$c_{SM}(\tau_{3,3,0})$ & 9 & 36 & 84 & 126 & 126 & 84 & 36 & 9 & 1 \\
$c_{SM}(\tau_{3,3,1})$ & 9 & 36 & 78 & 108 & 96  & 54 & 18 & 3 & 0 \\
$c_{SM}(\tau_{3,3,2})$ & 9 & 18 & 24 & 18  & 6   &  0 &  0 & 0 & 0\\
\end{tabular} \\
\vskip .2in
\item $\tau_{3,3,k}^\circ$ \\ 
\begin{tabular}{c | ccccccccc}
  Table  & $\bbP^{0}$ & $\bbP^{1}$  & $\bbP^{2}$  & $\bbP^{3}$  & $\bbP^{4} $  & $\bbP^{5} $ & $\bbP^{6}$  & $\bbP^{7}$ & $\bbP^{8}$ \\
\hline
$c_{SM}(\tau_{3,3,0}^{\circ})$ & 0 & 0  & 6  & 18 & 30 & 30 & 18 & 6 & 1\\
$c_{SM}(\tau_{3,3,1}^{\circ})$ & 0 & 18 & 54 & 90 & 90 & 54 & 18 & 3 & 0\\
$c_{SM}(\tau_{3,3,2}^{\circ})$ & 9 & 18 & 24 & 18 & 6  & 0  &  0 & 0 & 0\\
\end{tabular} \\
\end{enumerate}
\end{small}

\subsubsection{$m=n=4$} 
\begin{enumerate}
\item $\tau_{4,4,k}$ \\
\begin{tabular}{c | *{8}{c}}
  Table  & $\bbP^{0}$ & $\bbP^{1}$  & $\bbP^{2}$  & $\bbP^{3}$  & $\bbP^{4} $  & $\bbP^{5} $ & $\bbP^{6}$  & $\bbP^{7}$ \\
\hline
$c_{SM}(\tau_{4,4,0})$ &16 &120 &560 &1820 &4368 &8008 &11440 &12870 \\
$c_{SM}(\tau_{4,4,1})$ &16 &120 &560 &1796 &4224 &7528 &10360 &11114  \\
$c_{SM}(\tau_{4,4,2})$ &16 &120 &464 &1220 &2304 &3208 &3336  &2602   \\
$c_{SM}(\tau_{4,4,3})$ &16 &48  &104 &152  &144  &80   & 20   & 0     \\
\end{tabular} \\
\vskip .1in
\begin{tabular}{c | *{8}{c}}
  Table  & $\bbP^{8}$ & $\bbP^{9}$ & $\bbP^{10}$ & $\bbP^{11}$ & $\bbP^{12}$ & $\bbP^{13}$ & $\bbP^{14}$ & $\bbP^{15}$   \\
\hline
$c_{SM}(\tau_{4,4,0})$ &11440 &8008 &4368 &1820 &560 &120 &16 & 1  \\
$c_{SM}(\tau_{4,4,1})$ &9312  &6056 &3008 &1108 &288 &48  & 1 & 0 \\
$c_{SM}(\tau_{4,4,2})$ &1504  &616  &160  & 20  & 0  & 0  & 0 & 0 \\
$c_{SM}(\tau_{4,4,3})$ & 0    & 0   & 0   & 0   & 0  & 0  & 0 & 0\\
\end{tabular}

\vskip .2in
\item $\tau_{4,4,k}^\circ$ \\
\begin{tabular}{c | *{8}{c}}
  Table  & $\bbP^{0}$ & $\bbP^{1}$  & $\bbP^{2}$  & $\bbP^{3}$  & $\bbP^{4} $  & $\bbP^{5} $ & $\bbP^{6}$  & $\bbP^{7}$ \\
\hline
$c_{SM}(\tau_{4,4,0}^{\circ})$ &0 &0  &0   &24   &144  &480  &1080 &1756  \\
$c_{SM}(\tau_{4,4,1}^{\circ})$ &0 &0  &96  &576  &1920 &4320 &7024 &8512  \\
$c_{SM}(\tau_{4,4,2}^{\circ})$ &0 &72 &360 &1068 &2160 &3128 &3316 &2602 \\
$c_{SM}(\tau_{4,4,3}^{\circ})$ &16&48 &104 &152  &144  &80   &20   & 0   \\
\end{tabular} \\
\vskip .1in
\begin{tabular}{c | *{8}{c}}
  Table  & $\bbP^{8}$ & $\bbP^{9}$ & $\bbP^{10}$ & $\bbP^{11}$ & $\bbP^{12}$ & $\bbP^{13}$ & $\bbP^{14}$ & $\bbP^{15}$   \\
\hline
$c_{SM}(\tau_{4,4,0}^{\circ})$ &2128  &1952 &1360 &712  &272 &72  &12 & 1  \\
$c_{SM}(\tau_{4,4,1}^{\circ})$ &7808  &5440 &2848 &1088 &288 &48  &4  & 0    \\
$c_{SM}(\tau_{4,4,2}^{\circ})$ &1504  &616  &160  &20   & 0  & 0  &0  & 0\\
$c_{SM}(\tau_{4,4,3}^{\circ})$ & 0    &0    & 0   & 0   & 0  &0   & 0 & 0\\
\end{tabular}
\end{enumerate}

These examples suggest that the coefficients appearing in the $c_{SM}$ classes of
$\tau_{m,n,k}$ and $\tau^\circ_{m,n,k}$ are non-negative. We conjecture that these
classes are effective (see Conjecture \ref{conjecture; effective}).

The examples also indicate that several coefficients of the class for $\tau^\circ_{m,n,k}$
vanish. More precisely, the terms of dimension $0,1,\cdots, n-1-k$ in $c_{SM}(\tau_{m,n,k}^\circ)$ are $0$ (see Conjecture \ref{conjecture; vanishing}).

\subsection{Projectivized Conormal Cycles (Intersection cohomology characteristic cycles)}
\label{S; proconexample}
\begin{enumerate}
\item $\tau_{4,4,s}$ \\
\begin{small}
\begin{tabular}{c | *{7}{c}}
 Table & $h_1^{15}h_2$ & $h_1^{14}h_2^2$ & $h_1^{13}h_2^3$ & $h_1^{12}h_2^4$ & $h_1^{11}h_2^5$ & $ h_1^{10}h_2^6 $  & $h_1^9h_2^7$ \\
\hline
$Con(\tau_{4,4,1})$     & 0  & 0 & 0& 0& 0 & 0 & 0  \\
$Con(\tau_{4,4,2})$     & 0  & 0 & 0 & 20 & 80 & 176 & 256 \\
$Con(\tau_{4,4,3})$     & 4  & 12 & 36 & 68 &  84 &  60 &  20    \\
\end{tabular}
\vskip .05in
\begin{tabular}{c | *{8}{c}}
 Table &$ h_1^8h_2^8$ & $h_1^7h_2^9$ & $h_1^6h_2^{10}$ & $h_1^5h_2^{11}$ & $h_1^4h_2^{12}$ & $h_1^3h_2^{13} $ & $h_1^2h_2^{14}$ & $h_1h_2^{15}$ \\
\hline
$Con(\tau_{4,4,1})$    & 0   & 20 & 60 & 84 & 68 & 36 & 12 & 4     \\
$Con(\tau_{4,4,2})$    & 286   & 256 & 176 & 80  & 20 & 0  & 0  & 0 \\
$Con(\tau_{4,4,3})$    & 0     &   0 &  0  &  0  &  0 &  0 &  0 & 0    \\
\end{tabular}
\end{small}
\item $\tau_{5,4,s}$ \\ 
\begin{tiny}
\hspace*{-0.5cm}
\begin{tabular}{c | *{9}{c}}
  Table  & $h_1^{19}h_2$ & $h_1^{18}h_2^2$ & $h_1^{17}h_2^3$ & $h_1^{16}h_2^4$ & $h_1^{15}h_2^5$ & $ h_1^{14}h_2^6 $  & $h_1^{13}h_2^7$&$ h_1^{12}h_2^8$ & $h_1^{11}h_2^9$ \\
\hline
$Con(\tau_{5,4,1})$       & 0  & 0  & 0   & 0  & 0  & 0   & 0    & 0    &  0    \\
                         
$Con(\tau_{5,4,2})$       & 0  & 0  & 0   & 0  & 0  & 50 & 240 & 595 & 960  \\
                         
$Con(\tau_{5,4,3})$       & 0  & 10 & 40  & 105 & 176 & 190 & 120 & 35   & 0      \\
\end{tabular}
\vskip .05in 
\begin{tabular}{c | *{10}{c}}
\hspace*{-1cm}
  Table   & $h_1^{10}h_2^{10}$ & $h_1^9h_2^{11}$ & $h_1^8h_2^{12}$ & $h_1^7h_2^{13} $ & $h_1^6h_2^{14}$ & $h_1^5h_2^{15}$ & $h_1^4h_2^{16}$ & $h_1^3h_2^{17}$ & $h_1^2h_2^{18}$ & $h_1h_2^{19}$ \\
\hline
$Con(\tau_{5,4,1})$   &   0      & 0    & 35  & 120 & 190 & 176 & 105 & 40 & 10 & 0    \\
                         
$Con(\tau_{5,4,2})$   & 1116    & 960 & 595 & 240 & 50  & 0    & 0    & 0   & 0   & 0  \\
                         
$Con(\tau_{5,4,3})$   & 0        & 0    & 0    & 0    & 0    & 0    & 0    & 0   & 0   & 0  
\end{tabular}
\end{tiny}
\end{enumerate}
The examples suggest that $Con(\tau_{m,n,k})$ and $Con(\tau_{m,n,n-k})$ are symmetric to each other. This can be explained using the duality of determinantal varieties, as seen in \S\ref{S; dual}.

\subsection{Characteristic Cycles}
\begin{enumerate}
\item $\tau_{4,4,k}$ 
\vskip .1in
\begin{tabular}{c | *{8}{c}}
 Table & $h_1^{15}h_2$ & $h_1^{14}h_2^2$ & $h_1^{13}h_2^3$ & $h_1^{12}h_2^4$ & $h_1^{11}h_2^5$ & $ h_1^{10}h_2^6 $  & $h_1^9h_2^7$&$ h_1^8h_2^8$ \\
\hline
$Ch(\tau_{4,4,1})$        & 4  & 12 & 36 & 88 & 164 & 236 & 276 & 286 \\
$Ch(\tau_{4,4,1}^\circ)$  & 12 & 36 &108 &244 & 412 & 532 & 572 & 572 \\
$Ch(\tau_{4,4,2})$        & -8  & -24 & -72 & -156& -248 & -296 & -296 & -286 \\
$Ch(\tau_{4,4,2}^\circ)$  & -12 & -36 & -108& -224& -332 & -356 & -316 & -286 \\
$Ch(\tau_{4,4,3})$         & 4  & 12 & 36 & 68 &  84 &  60 &  20 & 0  \\
\end{tabular}
\vskip .05in
\begin{tabular}{c | *{7}{c}}
 Table & $h_1^7h_2^9$ & $h_1^6h_2^{10}$ & $h_1^5h_2^{11}$ & $h_1^4h_2^{12}$ & $h_1^3h_2^{13} $ & $h_1^2h_2^{14}$ & $h_1h_2^{15}$ \\
\hline
$Ch(\tau_{4,4,1})$        & 276 & 236 & 164 & 88 & 36 & 12 & 4\\
$Ch(\tau_{4,4,1}^\circ)$  & 532 & 412 & 244 &108 & 36 & 12 & 4\\
$Ch(\tau_{4,4,2})$        & -256 & -176 & -80  & -20 & 0  & 0  & 0\\
$Ch(\tau_{4,4,2}^\circ)$  & -256 & -176 & -80  & -20 & 0  & 0  & 0\\
$Ch(\tau_{4,4,3})$         &   0 &  0  &  0  &  0 &  0 &  0 & 0\\
\end{tabular}
\item $\tau_{4,3,k}$
\vskip .1in
\begin{tabular}{c | *{6}{c}}
  Table  & $h_1^{11}h_2$ & $h_1^{10}h_2^2$ & $h_1^9h_2^3$ & $h_1^8h_2^4$ & $h_1^7h_2^5$ & $h_1^6h_2^6$ \\
\hline
$Ch(\tau_{4,3,1})$       & 0  & 6 & 16 & 27 & 24 & 0       \\
$Ch(\tau_{4,3,1}^\circ)$ & 0  &12 & 32 & 54 & 48 & 10      \\
$Ch(\tau_{4,3,2})$       & 0  & -6 & -16 & -27 & -24 & -10     \\
\end{tabular}
\vskip .05in
\begin{tabular}{c | *{11}{c}}
  Table  & $h_1^5h_2^7$ & $h_1^4h_2^8$ &$h_1^3h_2^9$ & $h_1^2h_2^{10}$ & $h_1h_2^{11}$\\
\hline
$Ch(\tau_{4,3,1})$       & -24 & -27 & -16 & -6 & 0     \\
$Ch(\tau_{4,3,1}^\circ)$ & -24 & -27 & -16 & -6 & 0     \\
$Ch(\tau_{4,3,2})$       &  0  & 0 & 0 & 0 & 0     \\
\end{tabular}
\end{enumerate}

From the examples we can see that there are some interesting patterns on characteristic cycles, which weren't seen from $c_{SM}$ classes. For example, when $k=1$, the absolute values of the coefficients in $Ch(\tau_{m,n,1})$ are symmetric. For $m=n$, the coefficients of $Ch(\tau_{n,n,1})$ share the same sign, and hence are symmetric. This will be explained in the next section.

\subsection{Duality and Projectivized Conormal Cycle}
\label{S; dual}
Let $X\subset \bbP^n$ be a projective subvariety. Let $\bbP^{\vee n}$ be the space of hyperplanes in $\bbP^n$. This is also a projective space of dimension $n$.
We define the dual variety $X^\vee\subset \bbP^{\vee n}$ to be
\[
X^\vee=\overline{\{H | T_x X\subset H \text{ for some smooth point x}\}}
\]
In particular, the determinantal varieties are dual to each other:
\begin{thm}{\cite[Prop 1.7]{DE}}
Let $1\leq k<n$, and let $\tau_{m,n,k}\subset \bbP^{mn-1}$ be the determinantal variety. Then the dual of $\tau_{m,n,k}$ equals
\[
\tau_{m,n,k}^\vee=\tau_{m,n,n-k}\subset \bbP^{\vee mn-1}.
\]
In particular $\tau_{2k,2k,k}$ is self-dual.
\end{thm}

In fact the Chern-Mather classes of dual varieties can be mutually obtained by an involution.
\begin{thm}{\cite[Theorem 1.3]{Aluffi}}
Let $X\subset \bbP^n$ be a proper subvariety. 
Let $c_M(X)=\sum_{i=0}^n \beta_i H^i=(-1)^{\dim X}q(H)\cap [\bbP^n]$ be a polynomial $q$ of the hyperplane class $H$. Then 
\[
c_M(X^\vee)=(-1)^{\dim X^\vee} \mathcal{J}_n(q)  [\bbP^{\vee n}].
\]
Here $\mathcal{J}_n$ is an involution that sends a function $p=p(t)$ to
\[
\mathcal{J}_n(p)=p(-1-t)-p(-1)((1+t)^{n+1}-t^{n+1}).
\]
\end{thm}

For example, let $X=\tau_{4,4,1}$, it is dual to $\tau_{4,4,3}$. Then 
\begin{align*}
q(H)
=&(-1)c_M(\tau_{4,4,1}) \\
=&-48H^{15}-288H^{14}-1128H^{13}-3168H^{12}-6672H^{11}-10816 H^{10}-13716H^9 \\
&-13716H^8-10816H^7-6672H^6-3168H^5-1128H^4-288H^3-48H^2- 4H \\
\text{And} & \\
\mathcal{J}_n(q)
=& -16H^{15}-48H^{14}-104H^{13}-152H^{12}-144H^{11}-80H^{10}-20H^9 \\
=&(-1)c_M(\tau_{4,4,3}).
\end{align*}
Since the projectivized conormal cycle $Con(X)$ can be obtained from $c_M(X)$ via Prop~\ref{prop; chc_*}, the above theorem induces a symmetry between the projectivized conormal cycles of dual projective varieties. Let $\alpha$ be a class in $A_{n-1}(\bbP^n\times \bbP^n)$, then $\alpha=\sum_{i=0}^n \delta_i h_1^ih_2^{n-i}$ can be written as a class of the pull-back of hyperplane classes $h_1,h_2$ .
We define $\alpha^\dagger$ to be the `flip' of $\alpha$, i.e., $\alpha^\dagger=\sum_{i=0}^n \delta_i h_1^{n-i}h_2^i$. This `flip' process is compatible with addition: $(\alpha+\beta)^\dagger=\alpha^\dagger + \beta^\dagger$.

Let $X\subset \bbP^n$ be a $d$-dimensional projective subvariety. Let $Con(X)=\sum_{i=1}^n \delta_{i} h_1^{n-i}h_2^i \cap [\bbP^n\times \bbP^n]$ be a polynomial of pull-back hyperplane classes $h_1,h_2$. Then one has $Con(X^\vee)=Con(X)^\dagger$, i.e.,
\[
Con(X^\vee) =\sum_{i=1}^n \delta_{i} h_1^ih_2^{n-i} \cap [\bbP^{\vee n}\times \bbP^{n}].
\]
Moreover, the $l$th polar degree of $X$ equals the $(d-l)$th polar degree of $X^\vee$, and $gED(X)=gED(X^\vee)$.

For example, 
\begin{align*}
Con(\tau_{4,4,1})= 20h_1^7h_2^9 + 60 h_1^6h_2^{10}+ 84 h_1^5h_2^{11}+68 h_1^4h_2^{12}+36 h_1^3h_2^{13}+ 12 h_1^2h_2^{14}+ 4h_1^1h_2^{15}, 
\end{align*} and  
\begin{align*}
Con(\tau_{4,4,3})= 4h_1^{15}h_2^1+ 12h_1^{14}h_2^2+ 36h_1^{13}h_2^3+ 68h_1^{12}h_2^4+
84h_1^{11}h_2^5+ 60h_1^{10}h_2^6+ 20 h_1^{9}h_2^7 .
\end{align*}

Recall that $Con(X)=ch((-1)^{\dim X}Eu_X)$ and $Ch(X)=ch(\ind_X)$, then the relation between the local Euler obstruction and the indicator functions will also reflect on the characteristic cycles level. For determinantal varieties, Corollary~\ref{thm; csmalgorithm} shows that
\begin{align*}
\ind_{\tau_{m,n,k}}=\sum_{i=0}^{n-1-k} (-1)^i\binom{k+i-1}{k-1} Eu_{\tau_{m,n,k+i}}.
\end{align*} 
Thus we have
\begin{align*}
Ch(\tau_{m,n,k})
&=\sum_{i=0}^{n-1-k} (-1)^{(m+k+i)(n-k-i)-1+i} \binom{k+i-1}{k-1} Con(\tau_{m,n,k+i}) \\
&=\sum_{i=0}^{n-1-k} (-1)^{(m+k)(n-k)+i(m-n)-1} \binom{k+i-1}{k-1} Con(\tau_{m,n,k+i}). 
\end{align*}

In particular when $k=1$ we have the following property.
\begin{prop}
When $m$ is even and $n$ is odd, we have $Ch(\tau_{m,n,1})=-Ch(\tau_{m,n,1})^\dagger$. Otherwise we will have $Ch(\tau_{m,n,1})=Ch(\tau_{m,n,1})^\dagger$. 
\end{prop}
\begin{proof}
When $k=1$  we have $
Ch(\tau_{m,n,1})=\sum_{i=1}^{n-1} (-1)^{mn+i(m-n)} Con(\tau_{m,n,i}).
$
Notice that $Con(\tau_{m,n,i})=Con(\tau_{m,n,n-i})^\dagger$, thus 
\begin{align*}
(Con(\tau_{m,n,i})+Con(\tau_{m,n,n-i}))^{\dagger}=Con(\tau_{m,n,i})^\dagger +Con(\tau_{m,n,n-i})^\dagger =Con(\tau_{m,n,n-i})+Con(\tau_{m,n,i}).  
\end{align*}
is symmetric, while
\begin{align*}
(Con(\tau_{m,n,i})-Con(\tau_{m,n,n-i}))^{\dagger}=Con(\tau_{m,n,i})^\dagger -Con(\tau_{m,n,n-i})^\dagger =Con(\tau_{m,n,n-i})-Con(\tau_{m,n,i}).  
\end{align*}
is anti-symmetric.
When $m-n$ is even, 
\begin{align*}
&Ch(\tau_{m,n,1})
=\sum_{i=1}^{n-1} (-1)^{mn} Con(\tau_{m,n,i}) \\
=&Con(\tau_{m,n,1})+Con(\tau_{m,n,n-1}) + Con(\tau_{m,n,2})+Con(\tau_{m,n,n-2})+\cdots . 
\end{align*}
is a sum of symmetric terms, and hence is symmetric.
When $m-n$ is odd, but $n$ is even, then 
\begin{align*}
&(-1)^{mn}Ch(\tau_{m,n,1})
=\sum_{i=1}^{n-1} (-1)^{i} Con(\tau_{m,n,i}) \\
=&-Con(\tau_{m,n,1})-Con(\tau_{m,n,n-1}) + Con(\tau_{m,n,2})+Con(\tau_{m,n,n-2})-\cdots .  
\end{align*}
is an alternative sum of symmetric terms, and hence is symmetric. 
When $m-n$ is odd, and $n$ is odd, 
\begin{align*}
&(-1)^{mn}Ch(\tau_{m,n,1})
=\sum_{i=1}^{n-1} (-1)^{i} Con(\tau_{m,n,i}) \\
=&-Con(\tau_{m,n,1})+Con(\tau_{m,n,n-1}) + Con(\tau_{m,n,2})-Con(\tau_{m,n,n-2})-\cdots .  
\end{align*}
is an alternative sum of anti-symmetric terms, and hence is anti-symmetric.
\end{proof}

Also, Theorem~\ref{thm; eulerobstruction} shows that
\[
Eu_{\tau_{m,n,k}}=\sum_{i=k}^{n-1} \binom{i}{k} \ind_{\tau_{m,n,i}^\circ}.
\] 
Thus we have the following identity involving classes of Lagrangian cycles.
\begin{prop}
For $1\leq k<n$, we have the following identity 
\[
\sum_{i=k}^{n-1} \binom{i}{k} Ch(\tau_{m,n,i}^\circ)
=\sum_{i=n-k}^{n-1} \binom{i}{n-k} Ch(\tau_{m,n,i}^\circ)^\dagger
\]
\end{prop}
\begin{proof}
\begin{align*}
&\sum_{i=k}^{n-1} \binom{i}{k} Ch(\tau_{m,n,i}^\circ)=(-1)^{(m+k)(n-k)-1}Con(\tau_{m,n,k}) \\
=& (-1)^{(m+k)(n-k)-1}Con(\tau_{m,n,n-k})^\dagger 
= \sum_{i=n-k}^{n-1} \binom{i}{n-k} Ch(\tau_{m,n,i}^\circ)^\dagger . \qedhere
\end{align*}
\end{proof}
In particular, when $n=2k$, we have
\begin{align*}
&\sum_{i=k}^{2k-1} \binom{i}{k} Ch(\tau_{m,2k,i}^\circ)=\sum_{i=k}^{2k-1} \binom{i}{k} Ch(\tau_{m,2k,i}^\circ)^\dagger.
\end{align*}
Or
\begin{align*}
\sum_{i=k}^{2k-1} \binom{i}{k} (Ch(\tau_{m,n,i}^\circ)-Ch(\tau_{m,n,i}^\circ)^\dagger)=0.
\end{align*}
\begin{ex}
For example, let $m=n=4$, then as the reader may verify
by using the examples listed in \S\ref{S; proconexample}, we have
\begin{align*}
Ch(\tau_{4,4,1}^\circ)+2Ch(\tau_{4,4,2}^\circ)+3Ch(\tau_{4,4,3}^\circ)
=Ch(\tau_{4,4,3}^\circ)^\dagger  \/.
\end{align*}
and 
\begin{align*}
Ch(\tau_{4,4,2}^\circ)+3Ch(\tau_{4,4,3}^\circ)
=Ch(\tau_{4,4,2}^\circ)^\dagger + 3Ch(\tau_{4,4,3}^\circ)^\dagger  \/. 
\end{align*}
\end{ex}

\subsection{Conjectures}
\label{conjectures}
We conclude this paper by formulating two conjectures obtained from analyzing the examples. 
Assume $m\geq n$. Let $\tau_{m,n,k}\subset \bbP^{mn-1}$ be the degeneracy loci consisting of all $m$ by $n$ matrices with kernel dimension $\geq k$. Denote by $\tau_{m,n,k}^{\circ}=\tau_{m,n,k} \smallsetminus \tau_{m,n,k+1}$ the largest strata of $\tau_{m,n,k}$, and  write
\[
c_{SM}(\tau_{m,n,k}^{\circ})=\sum_{l=0}^{mn-1} \eta^{\circ}_l(m,n,k) H^{mn-1-l}=\sum_{l=0}^{mn-1} \eta^{\circ}_l(m,n,k)[\bbP^l]
\]
as a polynomial of degree $mn-1$ in the hyperplane class $H$.
Then
\begin{conj}
\label{conjecture; effective}
The coefficients $\eta^{\circ}_l(m,n,k)$ are non-negative for $0\leq l\leq mn-1$.  
\end{conj}
\begin{conj}
\label{conjecture; vanishing}
For $l=0,1,2,\cdots, n-k-2$, we have
\[
\eta^{\circ}_l(m,n,k)=0
\]
\end{conj}

For $\tau_{m,n,k}$, we write its $c_{SM}$ class as
\[
c_{SM}(\tau_{m,n,k})=\sum_{l=0}^{mn-1} \eta_l(m,n,k) H^{mn-1-l}=\sum_{l=0}^{mn-1} \eta_l(m,n,k)[\bbP^l]
\]
Notice that $\tau_{m,n,k}=\sqcup_{i=k}^{n-1} \tau^{\circ}_{m,n,i}$, hence $c_{SM}(\tau_{m,n,k})=\sum_{i=k}^{n-1} c_{SM}(\tau^{\circ}_{m,n,i})$. 
In terms of coefficients, we have $\eta_l(m,n,k)=\sum_{i=k}^{n-1} \eta^{\circ}_l(m,n,i)$ and $\eta_l^{\circ}(m,n,k)=\eta_l(m,n,k)-\eta_l(m,n,k+1)$.
Therefore the above conjectures imply that :
\begin{enumerate}
\item The coefficients $\eta_l(m,n,k)$ are non-negative for $0\leq l\leq mn-1$.
\item For $l=0,1,2,\cdots, n-k-2$,
\[
\eta_l(m,n,k)=\eta_l(m,n,k+1).
\]
\end{enumerate}

Also, recall that (Theorem~\ref{thm; eulerobstruction})
\[
Eu_{\tau_{m,n,k}}=\sum_{i=0}^{n-1-k} \binom{k+i}{k} \ind_{\tau_{m,n,k+i}^\circ},
\]
then we have
\[
c_M(\tau_{m,n,k})=c_*(Eu_{\tau_{m,n,k}})=\sum_{i=0}^{n-1-k} \binom{k+i}{k} c_*(\ind_{\tau_{m,n,k+i}^\circ})=\sum_{i=0}^{n-1-k} \binom{k+i}{k} c_{SM}({\tau_{m,n,k+i}^\circ}).
\]
Thus the above conjecture implies that the coefficients in the expression of $c_M(\tau_{m,n,k})$ are also non-negative for $0\leq l\leq mn-1$.

\bibliography{eu}

\end{document}